\documentclass[reqno]{amsart}
\usepackage[scale=0.75, centering, headheight=14pt]{geometry}
\usepackage[utf8
]{inputenc}
\usepackage[T1]{fontenc}
\usepackage{lmodern}
\usepackage[english]{babel}
\usepackage{amsmath}

\usepackage{amsmath,amssymb,amsfonts,amsthm}
\usepackage{mathtools,accents}
\usepackage{mathrsfs}
\usepackage{xfrac}
\usepackage{array} 
\usepackage{aliascnt}
\usepackage{booktabs} 
\usepackage{array} 

\usepackage{verbatim} 
\usepackage{subfig} 

\usepackage{mathrsfs, dsfont}
\usepackage{amssymb}
\usepackage{amsthm}
\usepackage{amsmath,amsfonts,amssymb,esint}
\usepackage{graphics,color}
\usepackage{enumerate}
\usepackage{mathtools,centernot}
\usepackage{esint}
\usepackage{bbm}

\usepackage{microtype}
\usepackage{paralist} 
\usepackage{cases}
\usepackage[initials, alphabetic]{amsrefs}
\allowdisplaybreaks

\usepackage{braket}
\usepackage{bm}

\usepackage[citecolor=blue,colorlinks]{hyperref}
\addto\extrasenglish{}

\usepackage{enumerate}
\usepackage{xcolor}
\definecolor{CYAN}{named}{cyan}
\usepackage{aliascnt}

\makeatletter
\def\newaliasedtheorem#1[#2]#3{
  \newaliascnt{#1@alt}{#2}
  \newtheorem{#1}[#1@alt]{#3}
  \expandafter\newcommand\csname #1@altname\endcsname{#3}
}
\makeatother

\usepackage{indentfirst}

\usepackage{esint}
\usepackage{mathrsfs}
\usepackage{stmaryrd}

\makeatletter
\newsavebox{\measure@tikzpicture}

\newcommand{\setword}[2]{%
  \phantomsection
  #1\def\@currentlabel{\unexpanded{#1}}\label{#2}%
}

\renewcommand\labelenumi{(\roman{enumi})}
\renewcommand\theenumi\labelenumi

\newtheorem{theorem}{\bf Theorem}[section]
\newtheorem{remark}[theorem]{\bf{Remark}}
\newtheorem{definition}[theorem]{\bf Definition}
\newtheorem{assumption}[theorem]{\bf Assumption}
\newtheorem{lemma}[theorem]{\bf Lemma}
\newtheorem{proposition}[theorem]{\bf Proposition}
\newtheorem{corollary}[theorem]{\bf Corollary}

\newtheorem{question}[theorem]{\bf Question}
\newtheorem*{theorem*}{Theorem}
\newtheorem*{proposition*}{Proposition}

\newcommand{\e}{{\rm e}}
\newcommand{\eps}{\varepsilon}
\newcommand{\R}{\mathbb R}
\newcommand{\N}{\mathbb N}

\newcommand{\Z}{\mathbb Z}

\newcommand{\T}{\mathbb T}

\newcommand{\diam}{\operatorname{diam}}

\newcommand{\initial}{\operatorname{in}}

\newcommand{\id}{\operatorname{id}}
\newcommand{\BV}{\operatorname{BV}}
\newcommand{\Leb}{\mathcal{L}}

\newcommand{\Borel}{\mathcal{B}}
\newcommand{\Haus}{\mathcal{H}}

\newcommand{\metric}{\mathsf{d}}
\newcommand{\mix}{\operatorname{mix}}
\DeclareMathOperator{\loc}{loc}

\DeclareMathOperator{\ap}{ap}
\newcommand{\apD}{\ap\nabla}

\DeclareMathOperator{\GL}{GL}

\allowdisplaybreaks
\numberwithin{equation}{section}

\title[Lyapunov exponents, entropy and mixing for DiPerna-Lions flows]{Lyapunov exponents, entropy and mixing for DiPerna-Lions flows}

\author[Elia Bru\'e]{Elia Bru\'e}
\address{Bocconi University, Department of Decision Sciences, Via Serfatti 25, 20136 Milano, Italy}
\email{elia.brue@unibocconi.it}

\author[Maria Colombo]{Maria Colombo}
\address{Institute of Mathematics, EPFL, Station 8, 1015 Lausanne, Switzerland}
\email{maria.colombo@epfl.ch}

\author[Carl J. P. Johansson]{Carl Johan Peter Johansson}
\address{Institute of Mathematics, EPFL, Station 8, 1015 Lausanne, Switzerland}
\email{carl.johansson@epfl.ch}

\begin{document}

\begin{abstract}
 { 
The main goal of this work is to establish an asymptotic form of Bressan’s mixing conjecture. To this end, we develop an ergodic-theoretic framework for incompressible DiPerna–Lions flows. Lyapunov exponents are defined via an Oseledets-type decomposition and related to metric entropy through a version of Ruelle’s inequality in this low-regularity setting. These tools yield sharp bounds on asymptotic regularity propagation and mixing rates, leading to our main result.

}
\end{abstract}

\maketitle

\tableofcontents

\section{Introduction}

The objective of this paper is to study {\it lower bounds on mixing efficiency} in connection with Bressan mixing conjecture~\cite{AB03}. To this end, we start by developing classical concepts and results from ergodic theory and dynamical systems in the DiPerna-Lions framework~\cites{DPL89,A04}, focusing on {\it flows generated by Sobolev-regular velocity fields}, which is the natural class to consider in this setting. Specifically, we consider time-periodic, incompressible velocity fields with Sobolev regularity on the $d$-dimensional torus:
\begin{equation}\label{eq:vectorintro}
    b \in L^1_{\text{loc}}(\mathbb{R}_+; W^{1,p}(\mathbb{T}^d; \mathbb{R}^d)) \text{ for some } p \geq 1, \quad \text{div}\, b = 0 \, .
\end{equation}
Without loss of generality, we assume the time period is $\tau = 1$, i.e., $b(t + 1, x) = b(t, x)$ for almost every $t \in \mathbb{R}_+$ and $x \in \mathbb{T}^d$ with $d\ge 2$. Since the pioneering works~\cite{DPL89,A04}, the theory has evolved into a vibrant research field with significant applications to other domains~\cite{DBPEJ18, LAMCAF17, SBPB20}. 

\subsection{The Ergodic-Theoretic Perspective} 
In this work, we adopt an ergodic-theoretic perspective to investigate the dynamical properties of \textit{Regular Lagrangian flows}, 
\begin{equation}
\begin{cases}
    \frac{d}{dt} X_{t}(x) = b(t, X_{t}(x)), \\
    X_{0}(x) = x, \\
    (X_{t})_{\#} \mathcal{L}^d = \mathcal{L}^d,
\end{cases} \tag{RLF}
\end{equation}
where the first two conditions hold for $\mathcal{L}^d$-almost every $x \in \mathbb{T}^d$, and the last condition ensures that the flow map preserves the Lebesgue measure. Although trajectories driven by Sobolev velocity fields may be non-unique~\cite{BCDL2021,AK23}, foundational results~\cite{DPL89,A04} guarantee the uniqueness of Regular Lagrangian flows (RLF) within our framework. In contrast to Lipschitz velocity fields, RLF can exhibit highly singular behavior, lacking Sobolev or even fractional Sobolev regularity~\cite{GAGCAM19,GAGCAM19b,PEJ16}. Nonetheless, their weak regularity and differentiability properties \cite{CLBPLL04,SBNDN22,GCCDL08,FL18,EBQHN21} render them a rich and compelling class of dynamical systems, general enough to include many known examples (see Section~\ref{sec:FrameworkMotivations}) yet sufficiently regular to support meaningful well-posedness and applications.

\medskip
By the time periodicity of $b(t,x)$, the Regular Lagrangian flow at integer times acts as a measure-preserving map on the torus. Thus, we focus on the dynamical system $(M, \mathcal{B}, \mu, T)$, where $M = \mathbb{T}^d$ is the $d$-dimensional torus, $\mathcal{B}$ is the Lebesgue $\sigma$-algebra, $\mu = \mathcal{L}^d$ is the Lebesgue measure, and $T(x) = X_1(x)$ is a measure-preserving map. 
The central results of this work are summarized as follows:
\begin{enumerate}
    \item We introduce a notion of \textit{Lyapunov exponents} $\lambda_1(x) > \dots > \lambda_{k(x)}(x)$ and an Oseledets decomposition tailored to our singular framework. Weak differentiability results, such as those in~\cite{CLBPLL04,SBNDN22,GCCDL08}, play a pivotal role in this development.
    \item We investigate the \textit{metric entropy} $h_\mu(T)$ of the system, recovering the Ruelle inequality, which relates metric entropy to the integral of positive Lyapunov exponents, in our singular framework.
    \item We provide lower bounds on the top Lyapunov exponent in terms of the exponential mixing rate of certain observables. This study advances the understanding of \textit{Bressan's mixing conjecture}~\cite{AB03}, demonstrating its asymptotic validity.
\end{enumerate}

\subsection{Lyapunov Exponents}

The key to defining Lyapunov exponents in our framework lies in the results of Le Bris and Lions~\cite{CLBPLL04}, where the authors establish the well-posedness of the system
\begin{equation}\label{eqn:Wt}
\left\{
\begin{array}{l}
\frac{d}{dt} W_t(x) = \nabla b(t, X_t(x)) W_t(x), \\
W_0(x) = {\rm id}, \\
\end{array}
\right.
\quad x \in \mathbb{T}^d, \, t \geq 0,
\end{equation}
where $b \in L^1_{\mathrm{loc}}(\mathbb{R}_+; W^{1,1}(\mathbb{T}^d; \mathbb{R}^d))$ is incompressible and $X_t(x)$ is the associated Regular Lagrangian flow (RLF). When $b(t,x)$ is $1$-periodic in time, as in our framework, this system gives rise to a natural cocycle: for every positive integer $k$,
\begin{equation}
W_k(x) = W_{1}(X_{k-1}(x)) \cdots W_{1}(X_1(x)) W_{1}(x) \quad \text{for a.e.\ } x \in \mathbb{T}^d,
\end{equation}
This cocycle is log-integrable, as required for applying the classical Oseledets theorem (see Section~\ref{subsec:OseledetsMET}). Notably, Le Bris and Lions~\cite{CLBPLL04} showed that $W_t(x)$ serves as a {\it differential in measure} of $X_t(x)$ (see Section~\ref{subsec:DifferentiabilityOfRLF}), a notion weaker than weak differentiability~\cite{LAJM}. However, when the velocity field is Sobolev regular with $p > 1$, the RLF $X_t(x)$ is weakly differentiable~\cite{GCCDL08} and satisfies a quantitative Lusin-Lipschitz inequality~\cite{EBQHN21}. This upgrades $W_t(x)$ to a weak differential, enabling quantitative results on Lyapunov exponents, such as the Ruelle inequality (see Section \ref{sec:introRuelle}). Finally, the construction of Lyapunov exponents extends to $\BV$ velocity fields, thanks to recent results~\cite{SBNDN22}.

\begin{theorem}[Lyapunov Exponents]\label{thm:LyapunovSobolev}
Let $b \in L^1_{\mathrm{loc}}(\mathbb{R}_+; \mathrm{BV}(\mathbb{T}^d; \mathbb{R}^d))$ be a $1$-periodic, incompressible velocity field with associated Regular Lagrangian flow $X_{t}(x)$. There exists a measurable set $\Gamma \subseteq \mathbb{T}^d$ of full measure and a function $k : \Gamma \to \mathbb{N}_{\geq 1}$ such that, for all $x \in \Gamma$, there are
 \begin{equation}
\infty > \lambda_1(x) > \ldots > \lambda_{k(x)}(x) > - \infty \quad \text{  and direct sum decomposition }\mathbb{R}^d = E^1_x \oplus \dots \oplus E^{k(x)}_x.
 \end{equation}
 for which the following holds:
\begin{enumerate}
    \item The maps $x \mapsto k(x)$, $x \mapsto \lambda_i(x)$, and $x \mapsto E^i_x$ are measurable. \label{item:LyapunovSobolevMeasurabilityOfMaps}
    \item The maps are $X_{1}$-invariant, i.e., $k \circ X_{1} = k$, $\lambda_i \circ X_{1} = \lambda_i$, and $W_{1}(x) E^i_x = E^i_{X_{1}(x)}$. \label{item:LyapunovSobolevInvariance}
    \item For all $v \in E^i_x \setminus \{0\}$, it holds
    \begin{equation}
        \lim_{t \to \infty} \frac{1}{t} \log |W_{t}(x) v| = \lambda_i(x).
    \end{equation} \label{item:LyapunovSobolevLyapunovExps}
\end{enumerate}
\end{theorem}

\subsection{Metric Entropy and the Ruelle Inequality}\label{sec:introRuelle}

The \textit{metric entropy} $h_\mu(T)$ of a measure-preserving dynamical system $(M, \mathcal{B}, \mu, T)$ quantifies the rate of increase in dynamical complexity as the system evolves over time~\cite{ANK59,YS59}. It is a central notion in modern ergodic theory, and we refer the reader to Section \ref{subsec:entropy} for our working definitions and to the textbooks~\cite{AKBH95,PW82} and surveys~\cite{LSY03,AK22} for comprehensive treatments.

A celebrated result by Ruelle~\cite{DR78} establishes that the metric entropy of a regular dynamical system is bounded by the sum of its positive Lyapunov exponents. Specifically, for a $C^1$ map $T: M \to M$ on a smooth manifold $M$ with an invariant measure $\mu$, Ruelle's theorem states that
\begin{equation}
h_\mu(T) \leq \int_M \sum_{i=1}^{k(x)} \lambda_i^+(x) m_i(x) \, d\mu(x), \label{eq:Ruelle}
\end{equation}
where $\lambda_i^+(x) = \max\{\lambda_i(x), 0\}$ are the positive Lyapunov exponents and $m_i(x)$ their multiplicities. When $T \in C^{1,\alpha}$ (for some $\alpha > 0$) and $\mu$ is equivalent to a Riemannian volume measure on $M$, Pesin's theorem~\cite{YP77} establishes equality in \eqref{eq:Ruelle}.

\bigskip
A central focus of this work is the study of metric entropy and the Ruelle inequality in the context of flows driven by Sobolev  velocity fields. Specifically, we consider the setting where $M = \mathbb{T}^d$, $\mu = \mathcal{L}^d$ (the Lebesgue measure), and the dynamics is induced by the Regular Lagrangian flow $X_{t}(x)$ associated with a $1$-periodic, incompressible velocity field $b \in L^1_{\mathrm{loc}}(\mathbb{R}_+; W^{1,p}(\mathbb{T}^d; \mathbb{R}^d))$. 

\begin{theorem}[Ruelle Inequality]\label{thm:RuelleRLF}
Let $X_t(x)$ be the RLF associated with a $1$-periodic, incompressible velocity field $b \in L^1_{\mathrm{loc}}(\mathbb{R}_+; W^{1,p}(\mathbb{T}^d; \mathbb{R}^d))$ with  $p > 1$.  Then
\begin{equation}
	h_{\mathcal{L}^d}(X_1) \leq \int_{\mathbb{T}^d} \sum_{i=1}^{k(x)} \lambda_i^+(x)\, m_i(x) \, dx,
\end{equation}
where $\lambda_i^+(x) := \max\{\lambda_i(x), 0\}$ are the positive Lyapunov exponents and $m_i(x)$ their multiplicities, as defined in Theorem~\ref{thm:LyapunovSobolev}.
\end{theorem}

 In Subsection~\ref{subsec:Ruelle}, we prove a more general version of Theorem~\ref{thm:RuelleRLF} (see Theorem~\ref{thm:RuelleGeneral}), applicable to a broad class of maps, including Regular Lagrangian flows (RLF) driven by $W^{1,p}$ incompressible velocity fields, classical Lipschitz and Sobolev maps, and maps with singularities as studied in~\cite{AKJMSFLFP86, CL95, NC99, NCRM06, VBMFDCL18, VBMFD20} to mention just a few.
To this end, we isolate a regularity assumption of the dynamical system which is shared by all these singular settings (see Assumption~\ref{assump:LusinLip} below).
To the authors' knowledge, {\it Theorem~\ref{thm:RuelleGeneral} represents the most general formulation of the Ruelle inequality when the volume measure is absolutelly contiuous to the volume measure, available in the literature}.

\bigskip

An open question, currently beyond the reach of existing techniques, concerns the finiteness of the metric entropy for flows driven by velocity fields in $\BV$ or $W^{1,1}$. This problem is closely tied to the propagation of regularity in such low-regularity settings and to Bressan's mixing conjecture.
Heuristically, the possibility of infinite entropy production in this framework may signal super-exponential trajectory separation on a set of positive measure, potentially pointing to new loss of regularity phenomena and perhaps even to the failure of Bressan’s conjecture.

\begin{question}[Entropy of $\BV$ flows]
Let $X_t(x)$ be the RLF associated with a divergence-free, $1$-periodic velocity field $b \in L^1_{\mathrm{loc}}(\mathbb{R}_+; \BV(\mathbb{T}^d; \mathbb{R}^d))$ (or even $b \in L^1_{\mathrm{loc}}(\mathbb{R}_+; W^{1,1}(\mathbb{T}^d; \mathbb{R}^d))$). 

Is the metric entropy $h_{\mathcal{L}^d}(X_1)$ finite? Does the Ruelle inequality remain valid in this setting?
\end{question}

\subsection{Log-Sobolev Regularity and Lyapunov Exponents}

The propagation of regularity for flow maps and passive scalars advected by singular velocity fields is a central topic of research with broad implications.

To formulate the problem, we consider a divergence-free, $1$-periodic velocity field $b \in L^1_{\mathrm{loc}}(\mathbb{R}_+; W^{1,p}(\mathbb{T}^d; \mathbb{R}^d))$ for some $p \geq 1$, and a bounded density function $\rho_{\rm in}\in L^\infty(\T^d)$. We study the Cauchy problem for the transport equation:
\begin{equation} \tag{TE}\label{eq:TransportEquation}
\left\{
\begin{array}{ll}
\partial_t \rho + b \cdot \nabla \rho = 0 & \text{in $\R_+ \times \mathbb{T}^d$}, \\
\rho(0, x) = \rho_{\mathrm{in}}(x) &  x\in \mathbb{T}^d.
\end{array}
\right.
\end{equation}
By results of DiPerna, Lions, and Ambrosio~\cite{DPL89, A04}, this system is well-posed, and the density $\rho\in L^\infty_t L^\infty_x$ is transported by the Regular Lagrangian flow $X_t(x)$.

In this framework, the {\it log-Sobolev norm} is the appropriate measure for studying regularity propagation:
\begin{equation}\label{eq:HomogeneousLogSobolevNorm}
\|\rho(t,\cdot)\|_{\dot{H}_{\mathrm{log}}(\mathbb{T}^d)}^2 :=  \int_{\mathbb{T}^d} \left( \int_{B_{1/5}(0)} \dfrac{|\rho(t,x+h) - \rho(t,x)|^2}{|h|^d} dh\right) dx .
\end{equation}
By~\cite{EBQHN21, PEJ16, FL18}, for $p > 1$, there exists a constant $C(d, p)$ such that
\begin{equation}\label{eq:standardpropagation}
\|\rho( t, \cdot)\|_{\dot{H}_{\mathrm{log}}(\mathbb{T}^d)}^2 
\leq
\|\rho_{\rm in}\|_{\dot{H}_{\mathrm{log}}(\mathbb{T}^d)}^2
+
C(d, p) \int_0^t \|\nabla b(s,\cdot)\|_{L^p} \, ds ,
\quad t\ge 0.
\end{equation}
Whether propagation holds for $p = 1$ remains an open question, with significant implications for Bressan’s mixing conjecture.

There are compelling reasons to focus on the log-Sobolev norm:
\begin{enumerate}
    \item No stronger norm can be propagated, as demonstrated in~\cite{GAGCAM19,GAGCAM19b,PEJ16,EBQHN21}. 
    \item The sharp growth of $\|\rho( t, \cdot)\|_{\dot{H}_{\mathrm{log}}(\mathbb{T}^d)}^2$ is linear in time, saturated when the system exhibits exponential separation of orbits.
\end{enumerate}
Given the second point, it is natural to hypothesize that the asymptotic growth rate $\frac{1}{t} \|\rho( t, \cdot)\|_{H_{\mathrm{log}}(\mathbb{T}^d)}$ is related to Lyapunov exponents. This is confirmed by the following theorem. The estimate is in terms of the {\it top Lyapunov exponent} defined as the maximum Lyapunov exponent in Theorem~\ref{thm:LyapunovSobolev} \ref{item:LyapunovSobolevLyapunovExps} or equivalently as 
\begin{equation}
\lambda_{\max}(x) = \lim_{t \to \infty} \frac{1}{t} \log |W_{t}(x)|.
\end{equation}

\begin{theorem}\label{thm:RLFAsymptoticRegularityThm}
Let $p \in (1, \infty]$, and let $b \in L^1_{\mathrm{loc}}(\mathbb{R}_+; W^{1,p}(\mathbb{T}^d; \mathbb{R}^d))$ be a $1$-periodic, divergence-free velocity field. Let $\rho_{\mathrm{in}} \in L^\infty \cap \mathrm{BV}(\mathbb{T}^d)$, and let $\rho \in L^{\infty}(\mathbb{R}_+ \times \mathbb{T}^d)$ be the unique solution to  \eqref{eq:TransportEquation}. Then
\begin{equation}\label{eqn:RLFAsymptoticRegularity}
	\limsup_{t \to \infty} \frac{1}{t} \|\rho(t,\cdot)\|_{\dot H_{\mathrm{log}}(\mathbb{T}^d)}^2 
	\leq C(d) \| \rho_{\mathrm{in}} \|_{L^\infty}^2 
	\int_{\mathbb{T}^d} \lambda_{\max}(x) \, dx
\end{equation}
where $\lambda_{\max}(x)$ denotes the top Lyapunov exponent of $X_1^{-1}$ defined in Theorem~\ref{thm:LyapunovSobolev}.
\end{theorem}

Theorem~\ref{thm:RLFAsymptoticRegularityThm} provides an asymptotic regularity estimate that improves upon the standard log-Sobolev bound in \eqref{eq:standardpropagation}. Indeed, we first observe that 
{by Fatou's lemma the top Lyapunov exponent of $X_1$ admits the following a priori bound in terms of the total variation of the velocity field:
\begin{equation}\label{eq:EstimateOnTopLyapunovExp}
\int_{\mathbb{T}^d} \lambda_{\max}(x) \, dx\leq \liminf_{t \to \infty} \dfrac{1}{t} \int_{\T^d} \log |W_{t}(x)| \, d x  \leq \fint_0^{1} \int_{\mathbb{T}^d} |\nabla b(t,x)| \, dx \, dt.
\end{equation}
The latter inequality follows from \eqref{eqn:Wt} for vector fields in $ L^1_{\mathrm{loc}}(\mathbb{R}_+; \mathrm{BV}(\mathbb{T}^d; \mathbb{R}^d))$ and will be justified rigorously for $\BV$ vector fields in Subsection~\ref{subsec:DifferentiabilityOfRLF}.}
Equation~\eqref{eq:EstimateOnTopLyapunovExp} also holds for the top Lyapunov exponent of $X_1^{-1}$ and yields
\begin{equation}
	\limsup_{t \to \infty} \frac{1}{t} \|\rho(t,\cdot)\|_{H_{\mathrm{log}}(\mathbb{T}^d)}^2 
	\leq C(d) \| \rho_{\mathrm{in}} \|_{L^\infty}^2 
	\fint_0^1 \int_{\mathbb{T}^d} |\nabla b(t,x)| \, dx \, dt.
\end{equation}

To the best of the authors' knowledge, this result is new even in the case of smooth velocity fields. We refer the reader to Theorem~\ref{thm:AsymptoticRegularityMPS} for a more general version of Theorem~\ref{thm:RLFAsymptoticRegularityThm}, which applies to singular maps as introduced in Section~\ref{sec:singmaps}.

\subsection{Bound on Mixing and Bressan's Conjecture}\label{sec:IntroMixing}
Mixing induced by incompressible flows is a fundamental stabilization mechanism in fluid dynamics, associated with the transfer of energy from large to small spatial scales. Although this process is conservative and reversible over finite time intervals, it leads to an effective irreversible loss of information in the long-time regime.

A widely used quantitative measure of the degree of mixing of a mean-free passive scalar $\rho(t,x)$, i.e., a solution to the transport equation~\eqref{eq:TransportEquation}, is given by the homogeneous negative Sobolev norm:
\begin{equation}\label{eq:H-1}
   \| \rho(t,\cdot) \|_{\dot{H}^{-1}(\T^d)}^2 = \sum_{k \in \Z^d \setminus \{ 0 \}} |k|^{-2} |\hat{\rho}(t,k)|^2,
\end{equation}
where $\hat{\rho}(t,k)$ denotes the $k$-th Fourier coefficient of $\rho(t,\cdot)$. This functional was introduced in~\cite{GMIMLP05,ZLJLTCRD11}, and it shares a close analogy with the classical notion of decay of correlations with H\"older observables, a well-established concept in mathematical physics and dynamical systems.

Two central questions in this context have received considerable attention: how rapidly the quantity in~\eqref{eq:H-1} can decay under constraints on the velocity field $b$, and which velocity fields can achieve such decay. It is classical that for incompressible $b \in L^\infty_t W^{1,\infty}_x$, the norm~\eqref{eq:H-1} cannot decay faster than exponentially, with rate $\sim e^{-t C \|\nabla b \|_{L^\infty}}$. This estimate was {improved }in~\cite{GCCDL08} (see also~\cite{GIAKXX14}), where the authors establish the exponential lower bound $\sim e^{-t C(p) \|\nabla b \|_{L^p}}$ for $b \in L^\infty_t W^{1,p}_x$, with $p \in (1,\infty]$.

It is conjectured that the optimal lower bound should be $\sim e^{-t C \|\nabla b \|_{L^1}}$, and that it holds for incompressible velocity fields in $W^{1,1}$ or $\BV$. This is closely related to Bressan’s rearrangement cost conjecture~\cite{AB03} and remains a major open problem. 

\medskip

A consequence of Theorem~\ref{thm:RLFAsymptoticRegularityThm} is a quantitative lower bound on mixing, expressed in terms of the time average of the maximal Lyapunov exponent.

\begin{theorem}\label{thm:RLFAsymptoticMixingThm}
Let $p \in (1, \infty]$, and let $b \in L^1_{\loc}(\mathbb{R}_+; W^{1,p}(\mathbb{T}^d; \mathbb{R}^d))$ be a $1$-periodic, divergence-free velocity field. Let $\rho_{\initial} \in L^\infty \cap \BV(\mathbb{T}^d)$ be mean-free, i.e., $\int_{\mathbb{T}^d} \rho_{\initial} = 0$. Let $\rho \in L^\infty(\mathbb{R}_+ \times \mathbb{T}^d)$ be the unique solution to \eqref{eq:TransportEquation}. Then
\begin{equation}\label{eq:MainEqOfRLFAsymptoticMixingThm}
    \liminf_{t \to \infty} \frac{1}{t} \log \| \rho(t,\cdot) \|_{\dot{H}^{-1}(\mathbb{T}^d)} 
    \ge -C(d) \left( \frac{\| \rho_{\initial} \|_{L^\infty}}{\| \rho_{\initial} \|_{L^2}} \right)^2
    \int_{\mathbb{T}^d} \lambda_{\max}(x) \, dx,
\end{equation}
where $\lambda_{\max}(x)$ denotes the top Lyapunov exponent of $X_1^{-1}$ defined in Theorem~\ref{thm:LyapunovSobolev}.
\end{theorem}

Theorem~\ref{thm:RLFAsymptoticMixingThm} establishes an asymptotic version of Bressan's mixing conjecture~\cite{AB03}. Indeed, by combining it with the estimate of the integral of the top Lyapunov exponents in terms of the $L^1$-norm of the gradient of the velocity field (see~\eqref{eq:EstimateOnTopLyapunovExp}), we obtain the following consequence: there exists a time $T > 0$, depending on the velocity field $b(t,x)$ and the initial datum $\rho_{\mathrm{in}}(x)$, such that
\begin{equation}\label{eqn:analytmixing}
    \| \rho(t,\cdot) \|_{\dot{H}^{-1}(\mathbb{T}^d)}
    \ge \exp\left( -C(d, \rho_{\mathrm{in}}) \, t \fint_0^1 \int_{\mathbb{T}^d} |\nabla b(t,x)| \, dx \, dt \right)
    \quad \text{for all } t \ge T.
\end{equation}

Although this result does not fully resolve Bressan's conjecture, due to the restriction $t \ge T$, it provides an asymptotic version of the conjectured lower bound.

\begin{remark} In Bressan's conjecture~\cite{AB03} and some related literature, see for instance \cite{GCCDL08, WC23, TEKLJM25}, mixing is measured in terms of the geometric mixing scale, for densities taking only values $ \pm 1$. 
We remark that \eqref{eqn:analytmixing} still holds if $  \| \rho(t,\cdot) \|_{\dot{H}^{-1}(\mathbb{T}^d)}$ is replaced by the geometric mixing scale of $\rho(t,\cdot)$ in this context, giving also in this phrasing an asymptotic version of Bressan's mixing conjecture. The notion of geometric mixing and the result are discussed in Subsection~\ref{subsec:AsympMixingGeneralAsymptoticMixingThmProof} and in Theorem~\ref{thm:AsymptoticMixingMPS}, Equation~\eqref{eq:GeometricMixingInThmForMaps} below.
\end{remark}

\medskip

Theorem~\ref{thm:RLFAsymptoticMixingThm} can also be interpreted as a quantitative lower bound on the maximal Lyapunov exponent, inferred from the exponential mixing rate of a single observable. Suppose there exists an initial datum $\rho_{\initial} \in L^\infty \cap \BV(\mathbb{T}^d)$ that undergoes exponential mixing:
\begin{equation}
    \| \rho(t,\cdot) \|_{H^{-1}} \le C e^{-t \beta}, \quad t \ge 0,
\end{equation}
for some $\beta > 0$. Then Theorem~\ref{thm:RLFAsymptoticMixingThm} implies 
\begin{equation}
    \int_{\mathbb{T}^d} \lambda_{\max}(x) \, dx 
    \ge C(d) \left( \frac{\| \rho_{\initial} \|_{L^2(\mathbb{T}^d)}}{\| \rho_{\initial} \|_{L^\infty(\mathbb{T}^d)}} \right)^2 \beta.
\end{equation}

In \cite[Proposition 1.4]{DDAKFRH24}, Dolgopyat, Kanigowski, and Rodriguez-Hertz establish a similar lower bound on the top Lyapunov exponent under the stronger assumption of \textit{universal exponential mixing}, where the decay estimate holds for every initial datum. 

In Section~\ref{sec:mixingandreg}, we prove a more general version of Theorem~\ref{thm:RLFAsymptoticMixingThm}, applicable to general measure-preserving maps (see Theorem~\ref{thm:AsymptoticMixingMPS}). In particular, our framework encompasses maps that may be more singular than those considered in \cite[Proposition 1.4]{DDAKFRH24}. However, we restrict attention to systems whose stationary measure is the Lebesgue measure $\Leb^d$.

\subsection{Connections to previous Literature}
In the fluid dynamics literature, tools from the theory of (stochastic) dynamical systems have been used in a multitude of contexts. Without aiming to be exhaustive, we highlight some of them:
\begin{itemize}
\item \emph{Stochastically forced fluid models.}
Lyapunov exponents and mixing have been studied in stochastically forced fluid models, see for instance \cite{JBABSPS22, JBABSPS21, WCKR24}.
Strongly related, concepts from passive scalar turbulence (see \cite{AMO49, SC51, GKB59}) have been investigated in \cite{JBABSPS22B, ABKKH24, WCKR25} mainly motivated by classical fluid turbulence.

\item \emph{Mixing in randomised velocity fields.}
Exponential mixing has been established for various stochastic velocity fields, e.g. the Pierrehumbert model \cite{RTP94} which alternates between vertical and horizontal sinusoidal shear flows where the phase shifts are random variables, see \cite{ABMCZRG23}. Many other stochastic velocity fields have been studied, see for instance \cite{MCZVNF24, VNFCS25}.

\item \emph{Mixing in low regularity velocity fields.}
Studying the time-one flow map of a 1-periodic deterministic velocity field, as in this manuscript, has been done previously. In \cite{TMEAZ19}, the authors construct a Sobolev velocity field whose time-one flow map is exponentially mixing. In \cite{TEKLJM25}, a Lipschitz velocity field having the same property is constructed. Properties of time-one flow maps of BV velocity fields are studied in ~\cite{SBMZ22, MZ22}, proving in particular that weak mixing is generic while mixing is not.
\end{itemize}

\subsection{Notation}
Throughout the manuscript, $C$ denotes a universal constant which may vary from line to line. We write $C(\ast)$ for a constant which depends only on the parameters $\ast$. 
We mostly work on the $d$-dimensional torus $\T^d$ with the geodesic distance $\metric$. For certain calculations, it is convenient to lift to the universal cover together with the Euclidean distance and when it is, we will do it tacitly, see for instance Lemma~\ref{lemma:UsefulLemmaAboutTheGradientBeingAGoodApprox}.

\subsection*{Acknowledgments}
MC and CJ were supported by the Swiss State Secretariat for Education, Research and Innovation (SERI) under contract number MB22.00034 through the project TENSE. 

\section{Framework and motivations}\label{sec:FrameworkMotivations}

We introduce a class of dynamical systems that includes flows generated by incompressible Sobolev velocity fields \ref{eq:vectorintro}. This class will be the primary focus of our work. In Section~\ref{sec:MapsSing}, we discuss its relationship with other singular frameworks previously studied in the theory of dynamical systems.
\subsection{Singular Maps}\label{sec:singmaps}

Let $(M, \metric)$ denote a metric measure space, equipped with the Borel $\sigma$-algebra $\mathcal{B}$ and a reference probability measure $\mu$. We consider a measure-preserving map $T: M \to M$, i.e., $T$ is measurable and satisfies
\begin{equation}
    T_{\#}\mu(A) := \mu(T^{-1}(A)) = \mu(A) \quad \forall A \in \mathcal{B}.
\end{equation}
Additionally, $T$ satisfies the following regularity assumption:

\begin{assumption}\label{assump:LusinLip}
There exists a function $g: M \to [0, \infty]$ such that:
\begin{enumerate}
    \item $g \in L^p(M, \mu)$ for some $p \geq 1$,
    \item For every $x, y \in M$, the following inequality holds:
    \begin{equation}\label{eq:LusinLipschitzInequalityInAssumption}
        \metric (T(x), T(y)) \leq e^{g(x) + g(y)} \metric(x, y).
    \end{equation}
\end{enumerate}
\end{assumption}

Even though Assumption~\ref{assump:LusinLip} is weaker than any Sobolev-type bound, it implies \textit{weak differentiability} of the map $T$ (see Section~\ref{sec:NotionsDifferentiabilityAndLyapunov}). Notice that Assumption~\ref{assump:LusinLip} is trivially satisfied when $T$ is Lipschitz regular.

\begin{remark}[Subadditivity Under Composition]

Consider two measure-preserving maps $T, S: M \to M$, each satisfying Assumption~\ref{assump:LusinLip} with functions $g_T, g_S \in L^p(M, \mu)$, respectively. Then, for all $x, y \in M$, the composition $T \circ S$ satisfies
\begin{equation}
    \metric((T \circ S)(x), (T \circ S)(y)) \leq \exp \big( g_T(S(x)) + g_T(S(y)) + g_S(x) + g_S(y) \big) \metric(x, y).
\end{equation}
Hence, $T \circ S$ satisfies Assumption~\ref{assump:LusinLip} with the function $g(x) = g_T(S(x)) + g_S(x)$, and the $L^p$-norm satisfies
\begin{equation}
    \| g \|_{L^p} \leq \| g_T \circ S \|_{L^p} + \| g_S \|_{L^p} = \| g_T \|_{L^p} + \| g_S \|_{L^p}.
\end{equation}
\end{remark}

 All the theorems stated in the introduction in terms of the regular Lagrangian flow will have a more general analogue in the sequel, stated in terms of {a dynamical system with} a map $T$ satisfying Assumption~\ref{assump:LusinLip} below. The correspondence of the various theorems is described in the following Table.

\begin{table}[h!]
\centering
\begin{tabular}{| c || c c c |}
\hline
  & Regular  & $\Leftarrow$ & General maps satisfying \\ 
  &Lagrangian Flows & & Assumption~\ref{assump:LusinLip} \\
 \hline  \hline
 Lyapunov exponents & Theorem~\ref{thm:LyapunovSobolev} & & Theorem~\ref{thm:LyapunovFirst} \\
 Estimates of entropy &  &  & Theorem~\ref{thm:BoundEntropyGeneral} \\  
 Ruelle's inequality & Theorem~\ref{thm:RuelleRLF} &  & Theorem~\ref{thm:RuelleGeneral} \\
 Asymptotic regularity & Theorem~\ref{thm:RLFAsymptoticRegularityThm} &  & Theorem~\ref{thm:AsymptoticRegularityMPS} \\
 Asymptotic mixing  & Theorem~\ref{thm:RLFAsymptoticMixingThm} &  & Theorem~\ref{thm:AsymptoticMixingMPS} \\
 \hline
\end{tabular}
\end{table}

For simplicity, throughout most of this work we restrict our attention to the specific case $M = \mathbb{T}^d$.

\subsection{Regular Lagrangian Flows}

The primary motivation for studying measure-preserving systems satisfying Assumption~\ref{assump:LusinLip} stems from the analysis of Regular Lagrangian flows associated with periodic velocity fields.
We consider a divergence-free velocity field $b \in L^1_{\text{loc}}(\mathbb{R}_+; \BV(\mathbb{T}^d; \mathbb{R}^d))$ on the torus of dimension $d\ge 2$.

\begin{definition}[Regular Lagrangian Flow \cite{A04}]
 A Borel map $X \colon \R_+ \times \R_+ \times \T^d \to \T^d$ is called a regular Lagrangian flow (RLF) associated with $b(t,x)$ if the following properties hold for all $s \in \R_+$:
 \begin{enumerate}
   \item for any $t \in \R_+$, the map $x \mapsto X(t,s,x)$ preserves the Lebesgue measure;
   \item for $\Leb^d$-a.e. $x \in \T^d$, the map $t \mapsto X(t,s,x)$ is an absolutely continuous curve which solves 
   \begin{equation}
    X(t,s,x) = x + \int_s^t b(r, X(r, s, x)) \, d r \quad \forall t \in \R_+.
   \end{equation}
 \end{enumerate}
\end{definition}
To shorten notation, we write $X_{t,s} \colon \T^d \to \T^d$ in place of $X_{t,s}(x) = X(t,s,x)$. When $s = 0$, we write $X_{t} = X_{t,0}$.
If $b$ is $1$-periodic in time, i.e., $b(t+1,x) = b(t,x)$ for almost every $x\in \T^d$ and $t\in \R_+$,  then for every integer $k\ge 1$ it holds
$X_{(k+1), k} = X_{1},$ {a.e. in }$\T^d$\, .
Hence, for every integer $n\ge 0$ we deduce
\begin{equation}
 X_{n, 0} = X_{1}\circ \ldots \circ X_{1} = (X_{1})^n ,\quad \text{a.e. in }\T^d\, .
\end{equation}

Additionally, the measure-preserving system $(\mathbb{T}^d, \mathcal{B}, \mathcal{L}^d, X_{1})$ satisfies Assumption~\ref{assump:LusinLip}. This was first established in~\cite{GCCDL08} using a slightly different formalism. We refer the reader to~\cite[Proposition 2.9]{EBQHN21} for a proof of the statement.

\begin{lemma}\label{lemma:LusinLipFlows}
 Let $p \in (1, \infty]$ and $b \in L^1([0,T] ; W^{1,p}(\T^d; \R^d))$ divergence-free. 
 Then there exists a Borel function $g:\T^d \to [0,\infty]$
 such that the following hold:
 \begin{equation}\label{eq:LusinLipflows}
     \begin{split}
       \| g \|_{L^p(\T^d)} &\leq C(d,p) \| \nabla b \|_{L^1([0, T]; L^{p}(\T^d; \R^{d \times d}))}
       \\
     e^{- g(x) - g(y)} &\leq \frac{\metric(X_{\tau,0}(x), X_{\tau,0}(y))}{\metric(x,y)} \leq e^{g(x) + g(y)} \quad \forall x \neq y \in \T^d, \, \tau \in [0, T].  
     \end{split}
 \end{equation}
 Here $X$ denotes the RLF generated by $b$, and $\metric(x,y)$ is the geodesic distance of the flat torus. 
\end{lemma}

In~\cite[Proposition 2.9]{EBQHN21}, the regularity estimate~\eqref{eq:LusinLipflows} was formulated in terms of the Euclidean distance in $\mathbb{R}^d$, identifying $b$ as a space-periodic velocity field in $\mathbb{R}^d$. However, it can be readily shown that this estimate implies the corresponding estimate in terms of the distance on the torus, as the Regular Lagrangian Flow (RLF) satisfies $X_{t,s}(x + k) - k = X_{t,s}(x)$ for every integer vector $k \in \mathbb{Z}^d$, a consequence of the space periodicity of $b$.

\begin{remark}\label{rmk:strongerAssumption}
From Lemma~\ref{lemma:LusinLipFlows}, it follows immediately that the map $T := X_{1}$ satisfies Assumption~\ref{assump:LusinLip}. Additionally, it satisfies the lower bound
\begin{equation}
    e^{-g(x) - g(y)} \metric(x, y) \leq \metric(T(x), T(y)) \quad \text{for all } x, y \in \mathbb{T}^d,
\end{equation}
which corresponds to the same regularity condition for the inverse map.
\end{remark}

\subsection{Maps with Singularities}
\label{sec:MapsSing}
Maps with singularities such as billiards have been subject of intense study (see for instance \cite{YGS70, AKJMSFLFP86, LBYGSNC90, CL95, LSY98, NC99, RM04, PBIPT08, NCZHK09, MDHZ14, VBMFDCL18, VBMFD20, VCMFDYLHZ24, EAYLMP24} and the references therein) and are strongly related to statistical mechanics.
In the book \cite{AKJMSFLFP86}, A. Katok and J.-M. Strelcyn study a class of maps with singularities. 
In particular, the authors prove Ruelle's inequality for such maps. 
They impose the following conditions: Let $(M, \metric)$ be a metric space and assume $M$ contains a dense Riemannian manifold $V$ of finite dimension. Consider an open subset $N \subseteq V$ and a differentiable map $T \colon N \to V$. Some further assumptions are made in \cite{AKJMSFLFP86} which we ignore here for the sake of simplicity. This map $T$ is the map which is studied. The singular set is
\[
 S = M \setminus N.
\]
Let $\mu$ be an invariant measure with respect to $T$, i.e. $T_{\#} \mu = \mu$.
Additionally, it is assumed that
\begin{enumerate}
 \item $\mu(B_{\eps}(S)) \leq C \eps^{a}$ for some $a > 0$; \label{item:KSLP-condition1} \medskip
 \item $\displaystyle\int_M \log^+ |\nabla T| \, d \mu < \infty$ and $\displaystyle\int_M \log^+ |\nabla (T^{-1})| \, d \mu < \infty$;
\end{enumerate} 
as well as other minor technical assumptions, which we omit here for simplicity. To prove Ruelle's inequality in this setting, the authors of~\cite{AKJMSFLFP86} additionally assume

\begin{enumerate} \setcounter{enumi}{2}
 \item $|\nabla T(x)| \leq C \metric(x, S)^{-q}$ for some $q \geq 1$. \label{item:KSLP-condition4}
\end{enumerate} 

The framework of \cite{AKJMSFLFP86} allows for domains with singularities, which we do not consider here for the sake of simplicity. Nevertheless, such cases appear to be treatable, as the DiPerna–Lions theory has been extended to metric measure spaces with synthetic curvature assumptions \cite{LADT14, EBDS20}. When restricted to singular maps on smooth manifolds, our regularity assumptions on the map significantly generalize those in \cite{AKJMSFLFP86}, as the following result demonstrates.

\begin{lemma}\label{lemma:RuelleWithSingularities}
Let $M$ be a smooth Riemannian manifold and $(M, \Borel, \mu, T)$ be a measure-preserving system such that \ref{item:KSLP-condition1} and \ref{item:KSLP-condition4} are satisfied. Then there exists $g \in \bigcap_{p \in [1, \infty)} L^p(M,\mu)$
such that
\begin{equation}\label{eq:AsymLusinLipForRuelle}
 |T(x) - T(y)| \leq \e^{g(x)} \metric(x,y) \quad \forall x,y \in M \setminus S.
\end{equation}
In particular, $(M, \Borel, \mu, T)$ satisfies Assumption~\ref{assump:LusinLip} for any $p \in [1,\infty)$.
\end{lemma}
\begin{proof}[Proof of Lemma~\ref{lemma:RuelleWithSingularities}]
We define the sets $U_0 = M$,  $U_j = B_{2^{-j}}(S)$ for all $j \geq 1$.
By \ref{item:KSLP-condition1} and \ref{item:KSLP-condition4}, we have that
\[
 \mu(U_j) \leq C 2^{-ja} \quad \text{and} \quad \| \nabla T \|_{L^{\infty}(M \setminus U_j)} \leq C 2^{jq}.
\]
We give now a natural definition for the function $g$: in each $U_j \setminus U_{j+1}$, it corresponds to the maximum possible value of $|\nabla T|$ there, up to an additive constant. More precisely, we set
\[
 g(x) =\log ( C d 2^{(j+3) q} )\quad \text{for every $x \in U_j \setminus U_{j+1}$,}
\]
We show that \eqref{eq:AsymLusinLipForRuelle} holds: to this end, let $x,y \in M \setminus S$, and without loss of generality assume that $x \in  U_j \setminus U_{j+1}$ for some $j \in \N$. Since from the definitions above we have that $B_{2^{-(j+2)}}(x) \subseteq M \setminus U_{j+2}$, for any $y \in B_{2^{-(j+2)}}(x)$, we have
\[
 \metric(T(x) , T(y)) \leq \| \nabla T \|_{L^{\infty}(B_{2^{-(j+2)}}(x))} d(x,y) \leq C 2^{(j+2)q} \metric(x,y).
\]
For any $y \in M \setminus B_{2^{-(j+2)}}(x)$, we have
\[
 \metric(T(x) , T(y)) \leq 2d \leq d 2^{j+3} \metric(x,y).
\]
Finally $g$ belongs to $L^p(M, \mu)$ for all $p \in [1, \infty)$ since
\begin{align*}
 \int_{M} |g(x)|^p \, d \mu(x) &= \sum_{j = 0}^{\infty} \int_{U_j \setminus U_{j+1}} |g(x)|^p \, d \mu(x) = \sum_{j = 0}^{\infty} \mu(U_j ) \left[ \log ( C 2^{(j+3)q}) \right]^p \\
 &= \sum_{j = 0}^{\infty} \mu(U_j) \left[ \log ( C ) + (j+2)q \log (2) \right]^p < \infty.
\end{align*}
\end{proof}

\begin{remark}
 It is clear from the proof above that the assumptions \ref{item:KSLP-condition1} and \ref{item:KSLP-condition4} can be relaxed while still satisfying Assumption~\ref{assump:LusinLip}. For instance, if we keep \ref{item:KSLP-condition1} identical, then there exists $\delta = \delta(a, p) > 0$ such that if \ref{item:KSLP-condition4} is replaced by
 \[
  |\nabla T (x)| \leq C \exp \left( d(x, S)^{- \delta} \right)
 \]
 then \eqref{eq:AsymLusinLipForRuelle} holds with some $g \in L^p(M, \mu)$.
\end{remark}

\section{Elements of Dynamical Systems and Ergodic Theory}
\label{sec:dynamical_systems}

We recall some key tools and concepts from the theory of dynamical systems and ergodic theory: the maximal ergodic function, the entropy, and the Oseledets multiplicative ergodic theorem.

In this section, $(M, \mathcal{A}, \mu)$ denotes an arbitrary probability space, and $(M, \mathcal{A}, \mu, T)$ represents a measure-preserving system, meaning that $\mu(T^{-1}(A)) = \mu(A)$ for all $A \in \mathcal{A}$.

\subsection{Maximal Ergodic Function}\label{subsec:MaxErgFct}

For every $g\in L^1(M,\mu)$, we define the maximal function
\begin{equation}
     g^{\star}(x) := \sup_{n \geq 1} \dfrac{1}{n} \sum_{i = 0}^{n-1} |g|(T^{i}(x)).
\end{equation}

A useful property of the maximal ergodic function is its almost invariance under the action of the measure-preserving map $T$: for all  $x \in M$
\begin{equation}
    g^\star(T(x)) = \sup_{n \geq 1} \frac{1}{n} \sum_{i=0}^{n-1} |g|(T^{i+1} x) = \sup_{n \geq 1} \frac{1}{n} \sum_{i=1}^{n} |g|(T^i x) \leq \sup_{n \geq 1} \frac{2}{n+1} \sum_{i=1}^{n} |g|(T^i x) \leq 2 g^\star(x).
\end{equation}

\begin{theorem}[Maximal Ergodic Theorem, {\cite[Theorem 2.24]{METW10}}]\label{thm:MaxErgThm}
 Let $(M, \mathcal{A}, \mu, T)$ be a measure-preserving system and $g \in L^1(M,\mu)$. Then for all $\lambda\ge 0$, we have
 \begin{equation}\label{eq:weakL1}
  \mu \left( \left\{ x \in M : g^{\star}(x) > \lambda \right\} \right) \leq \dfrac{\| g \|_{L^1(M,\mu)}}{\lambda}.
 \end{equation}
\end{theorem}

Interpolating \eqref{eq:weakL1} with $\| g^{\star} \|_{L^{\infty}(M,\mu)} \leq \| g \|_{L^{\infty}(M,\mu)}$, we obtain the following $L^p$-$L^p$ estimate.

\begin{corollary}\label{cor:MaxErgFct}
Let $p\in (1,\infty)$. For every $g\in L^p(M,\mu)$, the following estimate holds:
\begin{equation}
    \| g^\star\|_{L^p(M,\mu)} \le C(p) \| g\|_{L^p(M,\mu)}.
\end{equation}
\end{corollary}

\subsection{Metric Entropy}
\label{subsec:entropy}
We summarize some preliminary material on the metric entropy of measure-preserving dynamical systems. For a more comprehensive treatment of the subject, we refer the reader to~\cite{PW82, MV16}.

Given a measurable space $(M, \mathcal{A})$, a partition is a countable collection of pairwise disjoint measurable sets whose union is $M$. For each $x \in M$, we denote by $\xi(x)$ the unique element of $\xi$ containing $x$.

\begin{definition}
 Let $(M, \mathcal{A}, \mu)$ be a probability space and $\xi \subseteq \mathcal{A}$ a partition. The entropy of $\xi$ is 
 \begin{equation}
  H_{\mu}(\xi) = - \sum_{A \in \xi} \mu(A) \log(\mu(A))
 \end{equation}
 with the convention that $0 \log 0 = 0$.
\end{definition}

For any family of partitions $\xi_1, \dots, \xi_n$, we define their join as
\begin{equation}
    \bigvee_{i=1}^n \xi_i = \left\{ \bigcap_{i=1}^n A_i : A_i \in \xi_i, \quad i = 1, \dots, n \right\}.
\end{equation}

\begin{definition}[Metric Entropy]\label{def:Ent}
 Let $(M, \mathcal{A}, \mu, T)$ be a measure-preserving system and $\xi \subseteq \mathcal{A}$ a partition with finite entropy (i.e., $H(\xi) < \infty$).
 The entropy of $T$ with respect to $\xi$ is
 \begin{equation}\label{eq:entropy1}
  h_\mu(T, \xi) = \lim_{n \to \infty} \dfrac{1}{n} H_{\mu} \left( \bigvee_{i = 0}^{n-1} T^{-i} \xi \right) = \inf_{n \geq 1} \dfrac{1}{n} H_{\mu} \left( \bigvee_{i = 0}^{n-1} T^{-i} \xi \right),
 \end{equation}
 where $T^{-i} \xi = \{ T^{-i}(A) : A \in \xi \}$ denotes the pull-back partition. 
 Then, the entropy of $T$ is defined as
 \begin{equation}
  h_\mu (T) := \sup_{\xi} h_{\mu}(T, \xi)
 \end{equation}
 where the supremum is taken over all finite partitions.
\end{definition}

The metric entropy behaves well under iterations and inversions (when $T$ is invertible).

\begin{proposition}[{\cite[Proposition 9.1.14]{MV16}}]\label{prop:EntIterationsInversions}
 Let $(M, \mathcal{A}, \mu, T)$ be a measure-preserving system. Then $h_{\mu}(T^j) = j h_{\mu}(T)$ for every $j\ge 0$.
 Moreover, if $T$ is invertible, then $h_{\mu}(T^{-1}) = h_{\mu}(T)$.
\end{proposition}

Given two partitions $\xi$ and $\eta$, we say that $\xi$ is a finer than $\eta$ if every element of $\xi$ is contained in some element of $\eta$. Then, $\xi$ is called a refinement of $\eta$.

\begin{theorem}[Kolmogorov-Sinai,{ \cite[Theorem 9.2.1]{MV16}}]\label{thm:KolmogorovSinai}
Let $(M, \mathcal{A}, \mu, T)$ be a measure-preserving system. 
Let $\xi_1, \xi_2, \ldots$ be a sequence of partitions with finite entropy such that $\xi_{i+1}$ is a refinement of $\xi_i$ for all $i \geq 1$ and so that
\begin{equation}
 \mathcal{A} = \sigma \left( \bigcup_{i \geq 1} \xi_i \right) \quad \text{up to zero measure sets.}
\end{equation}
Then,
\begin{equation}
 h_{\mu}(T) = \lim_{n \to \infty} h_{\mu} (T, \xi_n).
\end{equation}
\end{theorem}

\begin{definition}[Conditional Entropy]
 Let $(M, \mathcal{A}, \mu, T)$ be a measure-preserving system and $\xi, \eta \subseteq \mathcal{A}$ two partitions with finite entropy. The conditional entropy of $\xi$ given $\eta$ is defined as
 \begin{equation}
  H_{\mu} (\xi | \eta) = \sum_{B \in \eta} \mu(B) \left( - \sum_{A \in \xi} \dfrac{\mu(A \cap B)}{\mu(B)} \log \left( \dfrac{\mu(A \cap B)}{\mu(B)} \right) \right)
 \end{equation}
 with the convention that $\frac{0}{0} = 0$.
\end{definition}
We will use the following property of conditional entropy coming from \cite[Lemma 9.1.5]{MV16}.
Given three partitions $\xi$, $\eta$, $\gamma$, we have
\begin{equation}\label{eq:UsefulPropertyOfEntropy}
 H_{\mu}(\xi \vee \eta | \gamma) = H_{\mu}(\xi | \gamma) + H_{\mu}(\eta | \xi \vee \gamma).
\end{equation}

\begin{lemma}\label{lem:UsefulEntLemma}
 Let $(M, \mathcal{A}, \mu, T)$ be a measure-preserving system and $\xi$ a partition with finite entropy. 
 Then,
\begin{equation}\label{eq:UsefulEntLemmaIneqs}
  h_\mu(T, \xi) \leq \int_M \log\big( \# \{ A \in \xi : \mu(T^{-1}(A) \cap \xi(x)) > 0 \}\big)
 \, d \mu.
 \end{equation}
 
\end{lemma}

\begin{proof}
First observe that for any $n \geq 1$
\begin{equation}
    H_{\mu} \left( \bigvee_{i = 0}^{n-1} T^{-i} \xi \right) {\leq} H_{\mu} \left( \bigvee_{i = 0}^{n-2} T^{-i} \xi \right) + H_{\mu} \left( T^{-(n-1)} \xi \big| T^{-(n-2)} \xi \right).
\end{equation}
where we used \eqref{eq:UsefulPropertyOfEntropy} with $\gamma$ the trivial partition and the fact that conditional entropy increases when conditioning with a less fine partition.
Noting that $H_{\mu} ( T^{-(n-1)} \xi | T^{-(n-2)} \xi ) = H_{\mu}(T^{-1} \xi | \xi)$ and arguing by induction, we deduce $h_{\mu}(T, \xi) \leq H_{\mu}(T^{-1} \xi | \xi)$. 

Therefore,
\begin{align*}
 H_\mu(T^{-1} \xi | \xi) &= \sum_{B \in \xi} \mu(B) \left( - \sum_{A \in \xi} \dfrac{\mu(T^{-1} (A) \cap B)}{\mu(B)} \log \left( \dfrac{\mu(T^{-1} (A) \cap B)}{\mu(B)} \right) \right)\\
 &\leq  \sum_{B \in \xi} \mu(B) \log (\# \{ A \in \xi: \mu(T^{-1} (A) \cap B) > 0 \} ) = \int_M \log \nu(x) \, d \mu(x).
\end{align*}
where in the inequality we used the concavity of the function $-x \log (x)$.
\end{proof}

\subsection{Oseledets Multiplicative Ergodic Theorem}\label{subsec:OseledetsMET}
Throughout this section,  let $(M, \mathcal{A}, \mu, T)$ be a measure-preserving system over a separable probability space.
 We state the multiplicative ergodic theorem, first proven by Oseledets \cite{VIO68}, see also \cite{DR79, LA02, MV14, LBYP23}.
Consider a mapping $A \colon M \to \R^{d \times d}$. The transformation form $M \times \R^d$ to itself defined by 
\[
 (x, v) \mapsto (T(x), A(x) v )
\]
is called the linear cocycle defined by $A$ over $T$. Iterating this transformation $n$ times, the second component gives
\[
 A^n(x) = A(T^{n-1}(x)) \ldots A(T(x)) A(x).
\]
Additionally, if $T$ is invertible and $A(x) \in \GL_d(\R)$ for $\mu$-a.e. $x$ then the inverse of the transformation $(x, v) \mapsto (T(x), A(x))$ is
\[
 (x, v) \mapsto (T^{-1}(x), A(x)^{-1} v ).
\]

\bigskip

\begin{theorem}[{\cite[Theorem 1.6 and Remark 1.8]{DR79}}]\label{thm:METOne}

 Let $A \colon M \to \R^{d \times d}$ be a measurable function. Assume that
 \begin{equation}\label{eq:LogIntegrabiliytOfCocycle}
  \int_M \log^+ |A(x)| \, d \mu(x) < \infty.
 \end{equation}
 Then, there exists a full measure set $\Gamma \subseteq M$ and a function $k \colon \Gamma \to \N_{\geq 1}$ such that for all $x \in \Gamma$ there exists:
 \begin{equation}
\infty > \lambda_1(x) > \ldots > \lambda_{k(x)}(x) \geq - \infty \text{ and a flag } \R^d = V^1_x \supsetneq \ldots \supsetneq V^{k(x)}_x \supsetneq \{ 0 \}
 \end{equation}
 for which the following holds:
 \begin{enumerate}
  \item the maps $x \mapsto k(x)$, $x \mapsto \lambda_i(x)$ and $x \mapsto V_x^{i}$ are defined on $\Gamma$ and are measurable;
  \item $k \circ T = k$, $\lambda_i \circ T = \lambda_i$ and $A(x) V_x^{i} \subseteq V_{T(x)}^{i}$;
  \item $\lim_{n \to \infty} \frac{1}{n} \log |A^n(x) v| = \lambda_i(x)$, $\forall v \in V_x^{i} \setminus V_x^{i+1}$ (with the convention that $V_x^{k+1} = \{ 0 \}$).
 \end{enumerate}
\end{theorem}

 The numbers $\lambda_1(x) > \ldots > \lambda_{k(x)}(x)$ and the flag $\R^d = V^1_x \supsetneq \ldots \supsetneq V^{k(x)}_x \supsetneq \{ 0 \}$ in the theorem above can be obtained as follows. For $\mu$-a.e. $x \in M$, the limit
 \begin{equation}\label{eq:ConvergenceSingularValuesToLyapunov}
  \Lambda(x) = \lim_{n \to \infty} \left( A^n(x)^T A^n(x) \right)^{\frac{1}{2n}}
 \end{equation}
 exists.
 Moreover, this matrix is symmetric positive semi-definite.
 Call $\exp \lambda_1(x) > \ldots > \exp \lambda_{k(x)}(x)$ the eigenvalues of $\Lambda(x)$ and $U^1_x, \ldots, U^{k(x)}_x$ the corresponding eigenspaces. Then, the flag is obtained as $V_x^{i} = U_x^{i} + \ldots + U_x^{k(x)}$. 
 {Call $\chi_i(U)$ the $i$-th singular value of $U \in \R^{d \times d}$, i.e. the $i$-th eigenvalue of $\sqrt{U^T U}$, counted with multiplicities. Then, since $\exp \lambda_i(x)$ is the limit of the $i$-th eigenvalue (up to multiplicities) of $( A^n(x)^T A^n(x) )^{1/2n}$, it follows from \eqref{eq:ConvergenceSingularValuesToLyapunov} that
 \begin{equation}\label{eq:SumOfPositiveLyapunovExpsAsLimit}
     \sum_{i = 1}^{k(x)} \lambda_i^{+}(x) m_i(x) 
     = \lim_{n \to \infty} \sum_{i = 1}^{d} \log^+ \left( [\chi_i(A^n(x))]^{1/n} \right) = \lim_{n \to \infty} \frac{1}{n} \sum_{i = 1}^{d} \log^+ \left( \chi_i(A^n(x)) \right).
 \end{equation}
}

\begin{remark}[{\cite[Remark 1.8]{DR79}}]
 We observe that if addition to \eqref{eq:LogIntegrabiliytOfCocycle}, we also have $A(x)$ $\mu$-a.e. invertible and
 \begin{equation}\label{ass:invert}
  \int_M \log^+ |A(x)^{-1}| \, d \mu(x) < \infty,
 \end{equation}
 then we may take $\Gamma$ in Theorem~\ref{thm:METOne} such that $\lambda_i(x) > - \infty$ and $A(x) V_x^{i} = V_{T(x)}^{i}$ for all $x \in \Gamma$.
\end{remark}

\begin{remark}[{\cite[Theorem 3.1]{DR79} and \cite[Theorem 4.2]{MV14}}]\label{rmk:StrongerMET}
    In the setting of Theorem~\ref{thm:METOne}, if additionally $T$ is invertible and $A(x) \in \GL_d(\R)$ for $\mu$-a.e. $x \in M$ with \eqref{ass:invert}
    then Theorem~\ref{thm:METOne} can be refined 
    to give the  existence of a measurable direct sum decomposition $\R^d = E_x^1 \oplus \ldots E_x^{k(x)}$ with $A(x) E_x^i = E^i_{T(x)}$ for $\mu$-a.e. $x \in M$ and all $i = 1, \ldots, k(x)$ such that
    \[
     \lim_{n \to \infty} \dfrac{1}{n} \log |A^n(x) v| =\lambda_i(x) \quad \forall v \in E_x^i \setminus \{ 0 \}.
    \]
    Moreover, angles between subspaces $E_x^i$ decay sub-exponentially along $\mu$-a.e. orbit.
\end{remark}

\section{Weak Differentiability and Lyapunov Exponents}\label{sec:NotionsDifferentiabilityAndLyapunov}

We introduce two well-known notions of differentiability: {\it approximate differentiability} and {\it differentiability in measure}. We then explain how these weak notions of differentiability can be used to define cocycles in our singular framework. Almost everywhere approximate differentiability has been established in \cite{GCCDL08} for regular Lagrangian flows (RLFs) associated with Sobolev velocity fields in $W^{1,p}$ for $p > 1$. Differentiability in measure was proven in \cite{CLBPLL04} for $W^{1,1}$ velocity fields, and extended to the $\BV$ setting in \cite{SBNDN22}. 

Thanks to the multiplicative ergodic theorem, these weak differentiability notions yield the existence of Lyapunov exponents. In particular, we establish Theorem~\ref{thm:LyapunovSobolev}.

\subsection{Approximate Limits}\label{subsec:ApproxLimits}

The upper and lower densities of a measurable set $E \subset \mathbb{T}^d$ at a point $x \in \mathbb{T}^d$ are defined as
\begin{equation}
    \limsup_{r \to 0^+} \frac{\mathcal{L}^d(B_r(x) \cap E)}{\mathcal{L}^d(B_r(x))} \quad \text{and} \quad \liminf_{r \to 0^+} \frac{\mathcal{L}^d(B_r(x) \cap E)}{\mathcal{L}^d(B_r(x))},
\end{equation}
respectively. When the upper and lower densities coincide, their common limit is called the density of $E$ at $x$. It is a classical result (see, for instance,~\cite[Theorem 1.35]{LCERFG15}) that the density exists and equals $1$ for $\mathcal{L}^d$-almost every $x \in E$. We refer to such points as the {\it full density points of $E$}.

\begin{definition}[Approximate Limit]
 Let $(N,\metric)$ be a metric space. We say that $F \colon \T^d \to N$ has an approximate limit at $x_0 \in \T^d$ if there exists a measurable set $E \subseteq \T^d$ whose density at $x_0$ is 1 such that
 \begin{equation}
  \lim_{x \to x_0, \, x \in E} F(x) \text{ exists.}
 \end{equation}
 Then, we denote it as $\ap \lim_{x \to x_0} F(x)$ and call it the approximate limit of $F$ at $x_0$.
\end{definition}

\subsection{Approximate Differentiability}
\label{subsec:ApproxDiff}

We introduce the concept of approximate differentiability.

\begin{definition}
Let $F: \mathbb{T}^d \to \mathbb{T}^d$ be a map and $x_0 \in \mathbb{T}^d$. The map $F$ is approximately differentiable at $x_0$ if there exists a matrix $A \in \mathbb{R}^{d \times d}$ such that
\begin{equation}
    \ap \lim_{x \to x_0} \frac{|F(x) - F(x_0) - A(x - x_0)|}{|x - x_0|} = 0.
\end{equation}
In this case, we write $\apD F(x_0) = A$ and call $\apD F(x_0)$ the approximate differential or approximate gradient of $F$ at $x_0$.
\end{definition}

When it exists, the approximate differential is unique. Additionally, any function $F \in C^1(\mathbb{T}^d; \mathbb{T}^d)$ is everywhere approximately differentiable, with its approximate differential coinciding with the classical differential, i.e., $\apD F(x) = \nabla F(x)$ for all $x \in \mathbb{T}^d$.

\begin{proposition}\label{prop:ApproxDiffPlusBoundPlusSmoothCovering}
Let $(\T^d, \Borel, \mu, T)$ be a measure-preserving system satisfying Assumption~\ref{assump:LusinLip} with $\mu \ll \Leb^d$. Then $T$ is $\mu$-a.e. approximately differentiable. In particular, the approximate differential satisfies 
\begin{equation}\label{eq:BoundOnApproximateGradient}
    |\apD T(x)| \leq \e^{3 g(x)} \quad \mu\text{-a.e.}
\end{equation} 
Moreover, there exists a collection of measurable sets $\{ A_i \}_{i \geq 1}$ and mappings $\Phi_i \in C^1(\overline{B}_i; M)$ such that
\begin{enumerate}
    \item $B_i$ are open balls containing $A_i$, and $\Phi_i = T$ on $A_i$,
    \item $\mu \left( \T^d \setminus \bigcup_{i \geq 1} A_i \right) = 0$.
\end{enumerate}
\end{proposition}

\begin{proof}
    Let $g \in L^p(\T^d, \mu)$ as in Assumption~\ref{assump:LusinLip}. Define the sets $C_i := \{ x \in \T^d : 2^i \leq g(x) \le  2^{i+1} \}$. Clearly, the restriction $T|_{C_i}$ is $e^{2^{i+2}}$-Lipschitz. Cover $C_i$ with a finite collection of balls $B_1, \dots, B_m$, each with radius $r \leq e^{-2^{i+10}} \cdot \text{inj}$, where $\text{inj} > 0$ denotes the injectivity radius of $\T^d$. For each set $A_{i,\ell} := C_i \cap \overline{B}_\ell$, both $A_{i,\ell}$ and its image $T(A_{i,\ell})$ are contained in open balls diffeomorphic to $\mathbb{R}^d$. By composing with an appropriate chart, we reduce the problem to studying a Lipschitz map in $\mathbb{R}^d$, which, by the standard Lusin approximation theorem, can be approximated by a $C^1$ function on a set of arbitrarily large measure. The details are left to the reader. From this it is clear that (i) and (ii) holds and in particular $T$ is $\mu$-a.e. approximately differentiable. It only remains to prove \eqref{eq:BoundOnApproximateGradient}. 
    Without loss of generality, up to adding an arbitrarily small constant to $g$, we assume $g>0$.
By definition of approximate differentiability, for a.e. $x$ it holds that
\begin{equation}\label{eq:ApproximateDifferentialAndIncrementalQuotient}
 |\apD T(x)|=\sup_{|v| = 1 } |\apD T(x) v| =\ap \limsup_{y \to x} \frac{|\apD T (x) (y-x)|}{|y-x|} \leq \ap \limsup_{y \to x} \frac{|T(y) - T(x)|}{|y-x|}
\end{equation}
For each $k \in \Z$, let $E_k = \{ 2^k < g \leq 2^{k+1} \}$ and note that $ \cup_{k \in \Z} E_k$ is a set of $\mu$-full measure. Let $k \in \Z$ be arbitrary and $x \in E_k$ a point of density 1. Then by \eqref{eq:ApproximateDifferentialAndIncrementalQuotient}
 \begin{equation}
   |\apD T(x)| \leq \ap \limsup_{y \to x} \e^{g(x) + 2^{k+1}} \leq \e^{g(x) + 2 g(x)} = \e^{3 g(x)}.
 \end{equation}
 Since $k \in \Z$ was arbitrary, we deduce that $|\apD T (x)| \leq \e^{3g(x)}$.
\end{proof}

\begin{proposition}\label{prop:ChainRuleForMPS}
 Let $(\T^d, \Borel, \mu, T)$ and $(\T^d, \Borel, \mu, S)$ be two measure-preserving systems with identical stationary measure $\mu \ll \Leb^d$. Assume that both measure-preserving systems satisfy Assumption~\ref{assump:LusinLip}. Then $T \circ S$ is $\mu$-a.e. approximately differentiable and it holds that
  \begin{equation}\label{eq:ChainRuleForMPS}
  \apD(T \circ S)(x) = (\apD T \circ S)(x) \apD S(x) \quad \text{for $\mu$-a.e. $x \in \T^d$.}
 \end{equation}
\end{proposition}

\begin{proof}
 The fact that $T \circ S$ is $\mu$-a.e. approximately differentiable follows from the fact that $(\T^d, \Borel, \mu, T \circ S)$ forms a measure-preserving system satisfying Assumption~\ref{assump:LusinLip}. It only remains to prove \eqref{eq:ChainRuleForMPS}. By Proposition~\ref{prop:ApproxDiffPlusBoundPlusSmoothCovering}, there exist two collections of balls $\{ B_i \}_{i \geq 1}$, $\{ B^{\prime}_j \}_{j \geq 1}$, two collections of measurable sets $\{ A_i \}_{i \geq 1}$, $\{ A_j^{\prime} \}_{j \geq 1}$, which both cover $\T^d$ up to zero $\mu$-measure sets, and mappings $\Phi_i \in C^1(\overline{B}_i; \T^d)$ $\Psi_j \in C^1(\overline{B}^{\prime}_j; \T^d)$ such that 
 \begin{align*}
  &A_i \subseteq B_i, \quad T = \Phi_i \text{ on $A_i$}, \quad A_j^{\prime} \subseteq B_j^{\prime}, \quad \text{and} \quad S = \Psi_j \text{ on $A_j^{\prime}$.}
 \end{align*}
 Since $T$ is $\mu$-a.e. approximately differentiable, we may assume that $T$ is approximately differentiable everywhere on $A_i$
 Since $S_{\#} \mu = \mu$, we get $\mu(\T^d \setminus \cup_{i \geq 1} S^{-1}(A_i)) = 0$.
 Fix some $i , j \geq 1$ and note that $T \circ S = \Phi_i \circ \Psi_j$ $\mu$ everywhere on $S^{-1}(A_i) \cap A_j^{\prime}$. Select a point $x \in S^{-1}(A_i) \cap A_j^{\prime}$ of density one such that $S(x)$ is also a point of density one w.r.t. $A_i$. Observe that $\apD T(S(x)) = \nabla \Phi_i(\Psi_j(x))$ and $\apD S(x) = \nabla \Psi_j(x)$. Thus
\begin{equation*}\label{eq:WhatToProveWithSmoothFunctions}
  \apD (T \circ S)(x) = \nabla (\Phi_i \circ \Psi_j)(x) = \nabla \Phi_i(\Psi_j(x)) \nabla \Psi_j (x) = \apD T(S(x)) \apD S(x).
 \end{equation*}
\end{proof}

We end the subsection with a result which will be useful later. Given a small $\eps > 0$, the lemma below returns a scale $\delta > 0$ and a set $A_{\eps}$ whose complement has measure less than $\eps$ and such that in $A_{\eps}$ and below the scale $\delta$, the approximate gradient provides a good approximation of the dynamics.

\begin{lemma}\label{lemma:UsefulLemmaAboutTheGradientBeingAGoodApprox}
 Let $(\T^d, \Borel, \mu, T)$ be a measure-preserving system satisfying Assumption~\ref{assump:LusinLip}. If $\mu \ll \Leb^d$, then for any $\eps > 0$ there exists a set $A_{\eps} \subseteq \T^d$ and $\delta \in (0,1)$ so that $\mu (\T^d \setminus A_{\eps}) < \eps$ and for all $x,y \in A_{\eps}$ with $|x-y| < \delta$, we have 
 \begin{equation}\label{eq:WhatToProveInUsefulLemmaAboutTheGradientBeingAGoodApprox}
  |T(y) - T(x) - \apD T(x) (y-x)| < \eps |x-y|.
 \end{equation}
\end{lemma}

\begin{proof}
Select $L \geq 1$ so that $\mu (\T^d \setminus \{ g \leq L \} ) \leq \eps / 2$. Then we note that $T$ is $2L$-Lipschitz on $\{ g \leq L \}$.
We prove that for $\mu$-a.e. $x \in \{ g \leq M \}$, there exists a $\delta_x > 0$ such that for all $y \in \{ g \leq M \} \cap B_{\delta_x}(x)$, we have
 \begin{equation}\label{eq:InClaimOfUsefulLemmabis}
  |T(y) - T(x) - \apD T(x)(y-x)| < \eps |x-y|.
 \end{equation}
 As in the proof of Proposition~\ref{prop:ApproxDiffPlusBoundPlusSmoothCovering}, by covering $\{ g \leq L \}$ with balls of sufficiently small radius and composing with an appropriate chart, we reduce the problem to studying a Lipschitz map on $\R^d$. By Rademacher's theorem, such a map is a.e. differentiable and in particular for any $\eps$ and any point of differentiability $x$, there exists $\delta_x > 0$ such that for all $y \in B_{\delta_x}(x)$, \eqref{eq:InClaimOfUsefulLemmabis} holds.
 To finish the proof, we notice that by taking $\delta > 0$ small enough with $A_{\eps} \coloneqq \{ g \leq L \} \cap \{ \delta_x > \delta \}$, we have $\mu(M \setminus A_{\eps}) < \eps$ and \eqref{eq:WhatToProveInUsefulLemmaAboutTheGradientBeingAGoodApprox} holds.
\end{proof}

\subsection{Lyapunov Exponents}
We employ Oseledets' theorem and the results of Section~\ref{subsec:ApproxDiff} on weak differentiability to establish the existence of Lyapunov exponents for systems satisfying Assumption~\ref{assump:LusinLip}.

\begin{theorem}\label{thm:LyapunovFirst}
 Let $(\T^d, \Borel, \mu, T)$ be a measure-preserving system with $\mu \ll \Leb^d$ and satisfying Assumption~\ref{assump:LusinLip}. 
 There exists a set $\Gamma \in \Borel$ of full $\mu$-measure and a function $k \colon \Gamma \to \N_{\geq 1}$ such that for all $x \in \Gamma$ there exists:
 \begin{itemize}
  \item numbers $\lambda_1(x) > \ldots > \lambda_{k(x)}(x)$ taking values in $\R \cup \{ - \infty \}$ and
  \item a flag $\R^d = V^1_x \supsetneq \ldots \supsetneq V^{k(x)}_x \supsetneq \{ 0 \}$
 \end{itemize}
 for which the following holds:
 \begin{enumerate}
  \item the maps $x \mapsto k(x)$, $x \mapsto \lambda_i(x)$ and $x \mapsto V_x^{i}$ are defined on $\Gamma$ and are measurable;
  \item $k \circ T = k$, $\lambda_i \circ T$ and $\apD T(x) V_x^{i} \subseteq V_{T(x)}^{i}$;
  \item we have 
  \begin{equation}
   \lim_{n \to \infty} \frac{1}{n} \log |\apD (T^n)(x) v| = \lambda_i(x), \quad \forall v \in V_x^{i} \setminus V_x^{i+1}
   \end{equation} 
   (with the convention that $V_x^{k+1} = \{ 0 \}$).
 \end{enumerate}
\end{theorem}

\begin{remark}
In the setting of Theorem~\ref{thm:LyapunovFirst} with the additional assumption that $T$ is invertible and $(\T^d, \Borel, \mu, T^{-1})$ is a measure-preserving system satisfying Assumption~\ref{assump:LusinLip}, then Theorem~\ref{thm:LyapunovFirst} can be refined to obtain the following stronger statement:
There exists a measurable direct sum decomposition $\R^d = E^1_x \oplus \ldots \oplus E^{k(x)}_x$ such that $\apD T(x) E^i_x = E^{i}_{T(x)}$ for $\mu$-a.e. $x \in \T^d$ and $\lim_{n \to \infty} \frac{1}{n} \log |\apD (T^n)(x)| = \lambda_i(x)$ for all $v \in E^{i}_x \setminus \{ 0 \}$.
Additionnally, the angles between subspaces $E^{i}_{x}$ decay sub-exponentially along $\mu$-a.e. orbit.
\end{remark}

\begin{proof}[Proof of Theorem~\ref{thm:LyapunovFirst}]
    Proposition~\ref{prop:ChainRuleForMPS} implies that the second component of linear cocycle defined by $\apD T$ over $T$ iterated $n$ times equals $\apD (T^n)$.
    Moreover, Equation~\eqref{eq:BoundOnApproximateGradient} from Proposition~\ref{prop:ApproxDiffPlusBoundPlusSmoothCovering} gives
\begin{equation}
    \int_{\T^d} \log^+ |\apD T| \, d \mu < \infty.
\end{equation}
Hence, Theorem~\ref{thm:METOne} can be applied to the system $(\T^d, \Borel, \mu, T)$ and the measurable mapping $\apD T$.
\end{proof}

\subsection{Differentiability in Measure of RLF's and Lyapunov Exponents}\label{subsec:DifferentiabilityOfRLF}
Whenever the integrability exponent $p$ of the Sobolev velocity field is strictly larger than one, then the RLF is $\Leb^d$-a.e. approximately differentiable.
The case $p = 1$ was considered by C.~Le Bris and P.-L.~Lions \cite{CLBPLL04}, who proved that the associated flow is differentiable in measure. Specifically, let $b \in L^1([0,T]; W^{1,1}(\T^d; \R^d))$ be a divergence-free velocity field, and let $X$ denote its regular Lagrangian flow.
 Then for all $0 \leq s \leq t \leq T$ the map $X_{t, s}$ is differentiable in measure. Precisely, for any $\delta > 0$, we have
 \begin{equation} 
  \lim_{r \to 0} \Leb^{2d} \left( \left\{ (x,v) \in \T^d \times B_1(0) : \left| \dfrac{X_{t,s}(x + rv) - X_{t,s}(x)}{r} - W_{t,s}(x) v \right| > \delta \right\} \right) = 0
 \end{equation}
 where 
 for all $s\ge 0$ and a.e. $x\in \T^d$, the map $t \mapsto W_{t,s}(x)\in \R^{d \times d}$ is an absolutely continuous solution to
 \begin{equation}\label{eq:OdeForDifferentialInMeasure}
  \left\{
  \begin{array}{l}
  \dfrac{d}{dt} W_{t,s}(x) = \nabla b (t, X_{t,s}(x)) W_{t,s}(x); \medskip \\
  W_{s,s}(x) = \id. \\
  \end{array}
  \right.
 \end{equation}
 
\begin{remark}[Approximate Differentiability vs. Differentiability in Measure]
As shown in \cite{LAJM}, differentiability in measure is a weaker notion than almost everywhere approximate differentiability.
\end{remark}

 More recently, the $\BV$ case has been treated by S.~Bianchini and N.~De~Nitti \cite{SBNDN22}, who extended the result of C.~Le~Bris and P.-L.~Lions to the class of divergence-free velocity fields $b \in L^1([0,T]; \BV(\T^d; \R^d))$. In fact, their result covers the more general case of vector fields with absolutely continuous divergence bounded from below. They replaced the system \eqref{eq:OdeForDifferentialInMeasure} with the following one, and showed that it admits a unique solution.

  \begin{equation}\label{eq:OdeForDifferentialInMeasureBV}
  \left\{
  \begin{array}{l}
  \dfrac{d}{dt} W_{t,s}(x) = (\nabla b)_{x,s}(dt) W_{t-,s}(x); \medskip \\
  W_{s,s}(x) = \id \\
  \end{array}
  \right.
 \end{equation}
 where $W_{t-,s}(x)$ denotes the left limit and $(\nabla b)_{x,s}$ is obtained through a disintegration (i.e. $(\nabla b)_{x,s}$ is a measure on $[0,T]$) along trajectories starting at time $s$, precisely
 \begin{equation}\label{eq:BianchiniDeNittiDisintegration}
     \int \int \varphi(t,x) (\nabla b)(dx \, dt) = \int \int \varphi(t, X_{t,s}(x)) (\nabla b)_{x,s}(dt) \, dx \quad \forall \varphi \in C^{\infty}([0,T] \times \T^d).
 \end{equation}
 In the $W^{1,1}$ Sobolev case, it is clear that
 \begin{equation}\label{eq:ChainRuleForDifferentialsInMeasureRLF}
     W_{t,s}(x) = W_{t,r}(X_{r,s}(x)) W_{r,s}(x) \quad \text{for all $0 \leq s \leq r \leq t \leq T$ and a.e. $x \in \T^d$}
 \end{equation}
 and
 \begin{equation}\label{eq:LogIntegrabilityForDifferentialsInMeasureRLF}
     \int_{\T^d} \log^+ |W_{t,s}(x)| \, dx \leq \int_s^t \int_{\T^d} |\nabla b| (dx \, dt).
 \end{equation}
 Indeed, this follows from the uniqueness of solutions to \eqref{eq:OdeForDifferentialInMeasure} together with Gr\"onwall's lemma. In the $\BV$ case, one observes that for the disintegration measure obtained via \eqref{eq:BianchiniDeNittiDisintegration}, we have
 \begin{equation}\label{eq:UsefulPropertyOfTheDisintegration}
     (\nabla b)_{x,s} = (\nabla b)_{X_{r,s}(s),r}
 \end{equation}
 for all $s \leq r$ and a.e. $x \in \T^d$.
 This follows from the fact that the flow is measure-preserving.
 Then \eqref{eq:ChainRuleForDifferentialsInMeasureRLF} follows from the uniqueness of solutions to \eqref{eq:OdeForDifferentialInMeasureBV} together with the fact that
 \begin{align*}
     \dfrac{d}{dt} \left[ W_{t,r}(X_{r,s}(x)) W_{r,s}(x) \right] &= (\nabla b)_{X_{r,s}(x),r}(dt) W_{t-,r}(X_{r,s}(x)) W_{r,s}(x) \\
     &\stackrel{\eqref{eq:UsefulPropertyOfTheDisintegration}}{=}
     (\nabla b)_{x,s}(dt) W_{t-,r}(X_{r,s}(x)) W_{r,s}(x).
\end{align*}
In addition, \eqref{eq:LogIntegrabilityForDifferentialsInMeasureRLF} follows from \cite[Theorem 4.1]{SBNDN22}. As for the RLF, we will write $W_{t} = W_{t,0}$.

\begin{proof}[Proof of Theorem~\ref{thm:LyapunovSobolev}]
    From the discussion above, we have
    \[
     W_{n}(x) = W_{1}(X_{1}^{n-1}(x)) \ldots W_{1}(X_{1}(x)) W_{1}(x) \quad and \quad \int_{\T^d} \log^+ |W_{1}(x)| \, dx \leq \int_0^1 \int_{\T^d} |\nabla b| (dx \, dt).
    \]
    Thus, the theorem follows from the multiplicative ergodic theorem, see Theorem~\ref{thm:METOne} and Remark~\ref{rmk:StrongerMET}.
\end{proof}

\begin{remark}
    If, in Theorem~\ref{thm:LyapunovSobolev}, one assumes the additional condition $p > 1$, then the measure-preserving system $(\T^d, \Borel, \Leb^d, X_{1})$ satisfies Assumption~\ref{assump:LusinLip}. Therefore, when $p > 1$, Theorems~\ref{thm:LyapunovFirst} and~\ref{thm:LyapunovSobolev} coincide for RLFs. Moreover, in this case, we have $W_1 = \apD X_1$.
\end{remark}

\section{Bounds on entropy}
 
The following result provides an upper bound on the entropy of maps satisfying Assumption~\ref{assump:LusinLip} for some $p > 1$. Notably, it holds even when the measure $\mu$ is singular with respect to the Lebesgue measure $\mathcal{L}^d$.

\begin{theorem}\label{thm:BoundEntropyGeneral}
Let $(\mathbb{T}^d, \mathcal{B}, \mu, T)$ be a measure-preserving dynamical system satisfying Assumption~\ref{assump:LusinLip} with some $p \in (1, \infty]$. Then, $T$ has finite metric entropy, and
\begin{equation}\label{eq:BoundEntropyGeneralFormula}
    h_\mu(T) \leq C(d, p) \| g \|_{L^p(\mathbb{T}^d, \mu)}.
\end{equation}
\end{theorem}

The idea of the proof  
is to divide $\T^d$ into a small grid of cubes and then further refine this partition into level sets $\{ 2^{i-1} \leq g^{\star} \leq 2^{i}\}$ where $g^{\star}$ is the maximal ergodic function; call this partition $\xi$. To estimate the entropy of $T$ with respect to $\xi$, for each $A \in \xi$, we estimate how many $B \in \xi$ can intersect $T(A)$. On one hand, this is determined by the smallest ball containing $T(A)$, whose radius can be estimated with $g^{\star}$. On the other hand, due to the almost-invariance of $g^{\star}$, if $A \subseteq \{g^{\star} \leq 2^{i} \}$ then $T(A) \subseteq \{g^{\star} \leq 2^{i+1} \}$, giving a further restriction on how many $B \in \xi$ may intersect $T(A)$.

\begin{proof}[Proof of Theorem~\ref{thm:BoundEntropyGeneral}]
Since the measure-preserving system $(\T^d, \Borel, \mu, T)$ satisfies Assumption~\ref{assump:LusinLip}, Subsection~\ref{subsec:MaxErgFct} implies that there exists a function $g^{\star} \in L^p(\T^d, \mu)$ such that $\| g^{\star} \|_{L^p(\T^d, \mu)} \leq C(p) \| g \|_{L^p(\T^d, \mu)}$,
\[
 \metric(T(x),T(y)) \leq \exp(g^{\star}(x) + g^{\star}(y)) \metric(x,y) \quad \forall x,y \in \T^d
\]
and $g^{\star}(Tx) \leq g^{\star}(x)$ for all $x \in \T^d$.
For any integer $k \geq 2$, we call $\xi_k$ the partition of $\T^d$ into $d$-dimensional open cubes whose sides have length $2^{-k}$. Note that $\# \xi_k = 2^{kd}$. For any integer $k \geq 2$, we call $\beta_k$ the partition of $\T^d$ into sets 
$$B_1^k = \{ g^{\star} < 2 \}, \ldots, B_i^k = \{ 2^{i-1} \leq g^{\star} < 2^{i} \}, \ldots, B_{k}^{k} = \{ 2^{k-1} \leq g^{\star} \}.$$ 
Note that $\Borel = \sigma(\cup_{k \geq 1} \beta_k \vee \xi_k )$ up to zero measure sets. 
Additionally, for all $k \geq 2$, $\beta_k \vee \xi_k$ are finite partitions and hence have finite entropy. Therefore, by the Kolmogorov-Sinai theorem (Theorem~\ref{thm:KolmogorovSinai})
\begin{equation}\label{eq:BoundEntropyGeneralApplicationOfKolmogorovSinai}
 h_{\mu}(T) = \lim_{k \to \infty} h_{\mu}(T, \beta_k \vee \xi_k).
\end{equation}
Now, fix an integer $k \geq 2$ and estimate $h_{\mu}(T, \beta_k \vee \xi_k)$. We define the function $\nu_k \colon \T^d \to \R$ as
\begin{equation}
 \nu_k (x) = \# \left\{ A \in \beta_k \vee \xi_k : \mu \Big( (\beta_k \vee \xi_k)(x) \cap T^{-1}(A) \Big) > 0 \right\}.
\end{equation}
Then, by Lemma~\ref{lem:UsefulEntLemma}, we find
\begin{equation}\label{eq:BoundEntropyGeneralApplicationOfUsefulLemma}
 h_{\mu}(T, \beta_k \vee \xi_k) \leq \int_{\T^d} \log \nu_k \, d\mu.
\end{equation}
\\
\textbf{Claim:} Let $x \in \T^d$ be arbitrary. Then if $g^{\star}(x) < 2^{k-1}$, 
\begin{enumerate}
 \item $(\beta_k \vee \xi_k)(x) \subseteq T^{-1}(\{ g^{\star} \leq 4 g^{\star}(x) \})$; \label{item:BoundEntropyGeneralClaimPointOne}
 \item $(\beta_k \vee \xi_k)(x) \subseteq T^{-1}(B_{d \e^{3 g^{\star}(x)} 2^{-k}}(T(x))$. \label{item:BoundEntropyGeneralClaimPointTwo}
\end{enumerate}
 We start by proving \ref{item:BoundEntropyGeneralClaimPointOne}. Assume that $g^{\star}(x) < 2^{k-1}$ and let $y \in (\beta_k \vee \xi_k)(x)$. Then by definition of $\beta_k$, it holds that $g^{\star}(y) \leq 2 g^{\star}(x)$. Hence, $g^{\star}(Ty) \leq 2 g^{\star}(y) \leq 4 g^{\star}(x)$. Thus \ref{item:BoundEntropyGeneralClaimPointOne} holds.
 We now prove \ref{item:BoundEntropyGeneralClaimPointTwo}. Again, assume $g^{\star}(x) < 2^{k-1}$ and remember that $g^{\star}(y) \leq 2 g^{\star}(x)$ for every $y \in (\beta_k \vee \xi_k)(x)$. Hence,
 \begin{equation}
  \metric(Tx,Ty) \leq \e^{g^{\star}(x) + g^{\star}(y)} \metric(x,y) \leq \e^{3 g^{\star}(x)} \metric(x,y) \leq \e^{3 g^{\star}(x)} \diam(\xi_k(x)) \leq d \e^{3 g^{\star}(x)} 2^{-k}.
 \end{equation}
 Thus \ref{item:BoundEntropyGeneralClaimPointTwo} holds.

 \medskip
Due to the claim above, if $g^{\star}(x) < 2^{k-1}$ then 
\begin{align*}
 \nu_k(x) &\leq \# \left\{ A \in \beta_k \vee \xi_k : \mu \Big( T^{-1}(\{ g^{\star} \leq 4 g^{\star}(x) \}) \cap T^{-1}(B_{d \e^{3 g^{\star}(x)} 2^{-k}}(T(x)) \cap T^{-1}(A) \Big) > 0 \right\}  \\
  &\leq \# \left\{ A \in \beta_k \vee \xi_k : \mu \Big( \{ g^{\star} \leq 4 g^{\star}(x) \} \cap B_{d \e^{3 g^{\star}(x)} 2^{-k}}(T(x)) \cap A \Big) > 0 \right\}. 
\end{align*}
We estimate the last term by counting how many $A \in \beta_k \vee \xi_k$ may have non-trivial intersection with $\{ g^{\star} \leq 4 g^{\star}(x) \} \cap B_{d \e^{3 g^{\star}(x)} 2^{-k}}(T(x))$. 
By definition, $\beta_k$ contains at most $1 + \log_2 (4 g^{\star}(x) + 1)$ many elements having non trivial intersection with $\{ g^{\star} \leq 4 g^{\star}(x) \}$.
Since $\xi_k$ is composed of cubes with sides of lenght $2^{-k}$, $\xi_k$ contains at most $C(d) (\e^{3 g^{\star}(x)} + 1)^d$ having non-trivial intersection with $B_{d \e^{3 g^{\star}(x)} 2^{-k}}(T(x))$. Thus
\begin{equation}
    \nu_k(x) \leq C(d) (\e^{3 g^{\star}(x)} + 1)^d \Big(1 + \log_2 (4 g^{\star}(x) + 1) \Big) \leq C(d) \e^{5 d g^{\star}(x)}.
\end{equation}
In the case where $g^{\star}(x) \geq 2^{k-1}$, we find
\begin{equation}
 \nu_k(x) \leq \# (\beta_k \vee \xi_k) \leq \# \beta_k \cdot \# \xi_k \leq k \cdot 2^{kd} \leq 2^{k(d+1)} \leq 2^{d+1} g^{\star}(x)^{d+1} \leq C(d) \e^{g^{\star}(x)}.
\end{equation}
Thus, $\nu_k(x) \leq C(d) \e^{5d g^{\star}(x)}$. Hence, 
\begin{equation}
 \int_{\T^d} \log \nu_k(x) \, d \mu(x) \leq 5d \int_{\T^d} g^{\star}(x) \, d \mu(x) + \log(C(d)) \leq C(d,p) \| g \|_{L^p( \T^d, \mu )} + \log(C(d)).
\end{equation}
Hence, due to \eqref{eq:BoundEntropyGeneralApplicationOfKolmogorovSinai} and \eqref{eq:BoundEntropyGeneralApplicationOfUsefulLemma}, we find
\begin{equation}\label{eqn:toimprove}
 h_{\mu}(T) = \lim_{k \to \infty}  h_{\mu}(T,  \beta_k \vee \xi_k) \leq C(d,p) \| g \|_{L^p( \T^d, \mu )} + \log(C(d)).
\end{equation}
To finish the proof, we improve \eqref{eqn:toimprove} getting rid of the second addend in the right-hand side. We observe that for any integer $j \geq 1$, the measure preserving-system $(\T^d, \Borel, \mu, T^j)$ satisfies Assumption~\ref{assump:LusinLip} with $g_j$ given by
\begin{equation}
 g_j = \sum_{\ell = 0}^{j - 1} g \circ T^{\ell}.
\end{equation}
Therefore, by \eqref{eqn:toimprove}, we find
\begin{equation}\label{BoundEntropyGeneralApplicationOfFirstBound}
 h_{\mu}(T^j) \leq C(d,p) \| g_j \|_{L^p( \T^d, \mu )} + \log(C(d)) = C(d,p) j \| g \|_{L^p( \T^d, \mu )} + \log(C(d))
\end{equation}
where we used the fact that $T$ is measure-preserving in the last equality. We observe, due to Proposition~\ref{prop:EntIterationsInversions}, that the left-hand side in \eqref{BoundEntropyGeneralApplicationOfFirstBound} equals $j h_{\mu}(T)$ and hence by dividing \eqref{BoundEntropyGeneralApplicationOfFirstBound} by $j$ and passing to the limit as $j \to \infty$, we find \eqref{eq:BoundEntropyGeneralFormula}. 
\end{proof}

 \subsection{Ruelle's Inequality}\label{subsec:Ruelle}
We prove Ruelle's inequality for measure-preserving systems $(\T^d, \Borel, \mu, T)$ satisfying Assumption~\ref{assump:LusinLip} with an invariant measure $\mu \ll \Leb^d$. 
\begin{theorem}\label{thm:RuelleGeneral}
 Let $(\T^d, \Borel, \mu, T)$ be a measure-preserving system satisfying Assumption~\ref{assump:LusinLip} with some $p \in (1, \infty]$. Assume $\mu \ll \Leb^d$. Then $T$ satisfies Ruelle's inequality, i.e. 
\begin{equation}\label{eq:RuelleGeneralFormulaInTheorem}
  h_{\mu}(T) \leq \int_{\T^d} \sum_{i = 1}^{k(x)} \lambda_i^+(x) m_i(x) \, d \mu(x).
 \end{equation}
\end{theorem}

\begin{proof}[Proof of Theorem~\ref{thm:RuelleRLF}]
It follows from Theorem~\ref{thm:RuelleGeneral} applied to the map $X_1$.  
Indeed, for $p > 1$, the measure-preserving system  $(\T^d, \Borel, \Leb^d, X_1)$ satisfies Assumption~\ref{assump:LusinLip} and $g \in L^p(\T^d)$ by Remark~\ref{rmk:strongerAssumption}.
\end{proof}

\begin{corollary}
    Let $b \in L^1_{\loc}(\R_+, W^{1,p}(\T^d, \R^d))$ be a $1$-periodic, divergence-free velocity field with associated regular Lagrangian flow $X_t(x)$. Then,
    \begin{equation}
        h_{\Leb^d}(X_1) \leq C(d) \fint_0^1 \int_{\T^d} |\nabla b(t,x)| \, dx \, dt.
    \end{equation}
\end{corollary}

\begin{proof}
This result follows immediately from Theorem~\ref{thm:RuelleRLF} and Equation~\eqref{eq:EstimateOnTopLyapunovExp}. Indeed,
\begin{equation}
 h_{\Leb^d}(T) \leq \int_{\T^d} \sum_{i = 1}^{k(x)} \lambda_i^+(x) m_i(x) \, d x \leq d \int_{\T^d} \lambda_{\max} \, dx \leq d \fint_0^1 \int_{\T^d} |\nabla b(t,x)| \, dx dt.
\end{equation}
\end{proof}

The set up of the proof is now standard \cite{RMSL12, QMXJSZS09}: the domain, here $\T^d$, is decomposed into small pieces forming a partition $\xi$ and for each element $A \in \xi$ of the partition, we estimate how many $B \in \xi$ can intersect $T(A)$. In the smooth setting, this depends on the singular values of $\nabla T$ in $A$ which eventually lead to Ruelle's inequality. In our singular setting, we employ Lemma~\ref{lemma:UsefulLemmaAboutTheGradientBeingAGoodApprox} to remove a small set, so that the approximate gradient describes the dynamics well in the remaining domain. 
In this bad small set, which is not invariant for the dynamics, we use the maximal ergodic function $g^{\star}$ in the same way as in the proof of Theorem~\ref{thm:BoundEntropyGeneral}, refining the partition according to the level sets of $g^{\star}$.

\begin{proof}[Proof of Theorem~\ref{thm:RuelleGeneral}]

Let $\eps > 0$ be arbitrary. By Lemma~\ref{lemma:UsefulLemmaAboutTheGradientBeingAGoodApprox}, there exists $A_{\eps} \subseteq \T^d$ and $\delta > 0$ such that $\mu(\T^d \setminus A_{\eps}) < \eps$ and for all $x,y \in A_{\eps}$ with $|x-y| < \delta$, we have 
\begin{equation}
 |T(y) - T(x) - \apD T(x) (y-x)| < |x-y|.
\end{equation}
For any integer $k \geq 1$, we call $\xi_k$ the partition of $\T^d$ into $d$-dimensional open cubes whose sides have length $2^{-k}$. Note that $\# \xi_k = 2^{kd}$. For any integer $k \geq 2$, we call $\beta_k$ the partition of $\T^d$ into the sets
\[
 A_{\eps}, \, B_1^k = \{ g^{\star} < 2 \} \cap A_{\eps}^c , \ldots, B_i^k = \{ 2^{i-1} \leq g^{\star} < 2^{i} \} \cap A_{\eps}^c, \ldots, B_{k}^{k} = \{ 2^{k-1} \leq g^{\star} \} \cap A_{\eps}^c.
\]
Note that $\Leb = \sigma(\cup_{k \geq 1} \beta_k \vee \xi_k )$ up to zero measure sets. 
Additionally, for all $k \geq 2$, $\beta_k \vee \xi_k$ are finite partitions and hence have finite entropy. Therefore, by the Kolmogorov-Sinai theorem (Theorem~\ref{thm:KolmogorovSinai})
\begin{equation}\label{eq:BoundEntropyGeneralApplicationOfKolmogorovSinaiInRuelle}
 h_{\mu}(T) = \lim_{k \to \infty} h_{\mu}(T, \beta_k \vee \xi_k).
\end{equation}
Now, fix an integer $k \geq 2$ and estimate $h_{\mu}(T, \beta_k \vee \xi_k)$. We define the function $\nu_k \colon \T^d \to \R$
\begin{equation}
 \nu_k (x) = \# \left\{ A \in \beta_k \vee \xi_k : \mu \Big( (\beta_k \vee \xi_k)(x) \cap T^{-1}A \Big) > 0 \right\}.
\end{equation}
{where, as above $\xi(x)$ denotes the element of the partition $\xi$ containing $x$.}
Then, by Lemma~\ref{lem:UsefulEntLemma}, we find
\begin{equation}\label{eq:BoundEntropyGeneralApplicationOfUsefulLemmaInRuelle}
 h_{\mu}(T, \beta_k \vee \xi_k) \leq \int_{\T^d} \log \nu_k \, d\mu.
\end{equation}
\\
\textbf{Claim: } Let $x \in \T^d$ be arbitrary. Then
\begin{enumerate}
 \item $(\beta_k \vee \xi_k)(x) \subseteq T^{-1} \left( B_{d2^{-k}} \Big( T(x) + \apD T(x) B_{d 2^{-k}} (0) \Big) \right) $ if $x \in A_{\eps}$; \label{item:ClaimInRuelleItem1}
 \medskip
  \item $(\beta_k \vee \xi_k)(x) \subseteq T^{-1}(\{ g^{\star} \leq 4 g^{\star}(x) \})$ if $g^{\star}(x) < 2^{k-1}$; \label{item:ClaimInRuelleItem2}
  \medskip
  \item $(\beta_k \vee \xi_k)(x) \subseteq T^{-1}( B_{d \e^{3g^{\star}(x)} 2^{-k}} (T(x)) )$ if $g^{\star}(x) < 2^{k-1}$. \label{item:ClaimInRuelleItem3}
\end{enumerate}

The proof of points \ref{item:ClaimInRuelleItem2} and \ref{item:ClaimInRuelleItem3} are proved in the same way as in the proof of the Claim in Theorem~\ref{thm:BoundEntropyGeneral}. Therefore, we only need to prove point \ref{item:ClaimInRuelleItem1}. Let $x \in A_{\eps}$ and $y \in (\beta_k \vee \xi_k)(x)$. Then
 \[
  |T(y) - T(x) - \apD T(x) (y-x)| < |x-y|
 \]
 so that $Ty \in B_{d 2^{-k}}(T(x) + \apD T(x) B_{d 2^{-k}}(0))$. This proves \ref{item:ClaimInRuelleItem1} and therefore we have proved the claim.

Using the first two claims, if $x \in A_{\eps}$ and $g^{\star}(x) < 2^{k-1}$ 
\begin{align*}
 \nu_k(x) &\leq\# \left\{ A \in \beta_k \vee \xi_k : \mu \Big( T^{-1} \left( B_{d2^{-k}} \Big( T(x) + \apD T(x) B_{d 2^{-k}} (0) \Big) \cap \{ g^{\star} \leq 4 g^{\star}(x) \} \cap A \right) \Big) > 0 \right\} \\
 &\quad \leq \# \left\{ A \in \beta_k \vee \xi_k : \mu \left( B_{d2^{-k}} \Big( T(x) + \apD T(x) B_{d 2^{-k}} (0) \Big) \cap \{ g^{\star} \leq 4 g^{\star}(x) \} \cap A \right) > 0 \right\}.
\end{align*}
We count how many $A \in \beta_k \vee \xi_k$ may have non-trivial intersection with $E\cap \{ g^{\star} \leq 4 g^{\star}(x) \},$ where $E:=B_{d2^{-k}} \Big( T(x) + \apD T(x) B_{d 2^{-k}} (0) \Big) $.
By definition, $\beta_k$ contains at most $1 + \log_2(4 g^{\star}(x) + 1)$ elements having non-trivial intersection with the set $\{ g^{\star} \leq 4 g^{\star}(x) \}$.
If an element of $\xi_k$ has non-trivial intersection with $E$, then this element is contained in the larger set $B_{d2^{-k}}(E)$.
Hence $\xi_k$ contains at most
$2^{k d}\Leb^d( B_{d2^{-k}}(E))$
such elements. A fact from linear algebra (see e.g. \cite[Lemma II.2.2]{QMXJSZS09}) implies $$
2^{k d}\Leb^d( B_{d2^{-k}}(E))=2^{kd} \Leb^d \left( B_{d2^{1-k}} \Big( T(x) + \apD T(x) B_{d 2^{-k}} (0) \Big) \right) \leq C(d) \prod_{i = 1}^d \max \{ \chi_i (\apD T(x)), 1 \}$$ 
where we recall that $\chi_i$ denotes the $i$-th singular value 
\begin{equation}
    \nu_k(x) \leq C(d) \prod_{i = 1}^d \max \{ \chi_i (\apD T(x)), 1 \} \cdot \Big( 1 + \log_2(4 g^{\star}(x) + 1) \Big) \leq C(d) \e^{\eps g^{\star}(x)} \prod_{i = 1}^d \max \{ \chi_i (\apD T(x)), 1 \}.
\end{equation}

If $x \not\in A_{\eps}$ and $g^{\star}(x) < 2^{k-1}$, then, as in the proof of Theorem~\ref{thm:BoundEntropyGeneral}, $\nu_k(x) \leq C(d) \e^{5d g^{\star}(x)}$. 
If $g^{\star}(x) \geq 2^{k-1}$, then as in the proof of Theorem~\ref{thm:BoundEntropyGeneral}, $\nu_k(x) \leq C(d) \e^{\eps g^{\star}(x)}$.
In conclusion,
\begin{equation}
 \nu_k(x) \leq C(d) \big(\mathbbm{1}_{A_{\eps}}(x) \e^{\eps g^{\star}(x)} \cdot \prod_{i = 1}^d \max \{ \chi_i (\apD T(x)), 1 \} +  \mathbbm{1}_{A_{\eps}^c}(x) \e^{5 d g^{\star}(x)} \big).
\end{equation}
We find 
\begin{align*}
 \int_{\T^d} \log \nu_k(x) \, d \mu(x) 
 &\leq\log(C(d)) + \int_{A_{\eps}} \left[ \eps g^{\star}(x) + \sum_{i = 1}^{k(x)} \log^+\Big(\chi_i(\apD T(x))\Big) \right] \, d \mu(x)  
 + 5 d \int_{\T^d \setminus A_{\eps}} g^{\star}(x) \, d \mu(x) \\
 &\leq \log(C(d)) + \eps \| g^{\star} \|_{L^1(\T^d, \mu)} + \int_{\T^d} \sum_{i = 1}^{k(x)} \log^+\Big(\chi_i(\apD T(x))\Big) d \mu(x) \\
 &\qquad + 5 d \mu(\T^d \setminus A_{\eps})^{1/p^{\prime}} \| g^{\star} \|_{L^p( \T^d, \mu )} \\
 &\leq \log(C(d)) + C(d)( \eps + 5d \eps^{1/p^{\prime}} ) \| g^{\star} \|_{L^p( \T^d, \mu )} + \int_{\T^d} \sum_{i = 1}^{k(x)} \log^+\Big(\chi_i(\apD T(x))\Big) d \mu(x).
\end{align*}
Hence,  by \eqref{eq:BoundEntropyGeneralApplicationOfKolmogorovSinaiInRuelle} and by \eqref{eq:BoundEntropyGeneralApplicationOfUsefulLemmaInRuelle}, we find
\begin{equation}
 h_{\mu}(T)  = \lim_{k \to \infty} h_{\mu}(T, \beta_k \vee \xi_k) \leq \log(C(d)) + C(d)( \eps + 5d \eps^{1/p^{\prime}} ) \| g^{\star} \|_{L^p( \T^d, \mu )} + \int_{\T^d} \sum_{i = 1}^{k(x)} \log^+\Big(\chi_i(\apD T(x))\Big) d \mu(x).
\end{equation}
Since $\eps > 0$ was arbitrary, we deduce
\begin{equation}\label{eq:AlmostRuelle}
 h_{\mu}(T) \leq \log(C(d)) + \int_{\T^d} \sum_{i = 1}^{k(x)} \log^+\Big(\chi_i(\apD T(x))\Big) d \mu(x).
\end{equation}
To finish the proof, we note that for all integers $j \geq 1$, $(\T^d, \Leb, \mu, T^j)$ is a measure-preserving system and satisfies Assumption~\ref{assump:LusinLip} with $g_j \colon \T^d \to \R$ given by
\begin{equation}
 g_j = \sum_{\ell = 0}^{j-1} g \circ T^{\ell}
\end{equation}
By \eqref{eq:AlmostRuelle}, we deduce
\begin{equation}\label{eq:AlmostRuelleAppliedToIterations}
 h_{\mu}(T^j) \leq \log(C(d)) + \int_{\T^d} \sum_{i = 1}^{k(x)} \log^+\Big(\chi_i(\apD (T^j)(x))\Big) d \mu(x).
\end{equation}
Noting that $h_{\mu}(T^j) = j h_{\mu}(T)$ and dividing \eqref{eq:AlmostRuelleAppliedToIterations} by $j$, we deduce
\begin{equation}
 h_{\mu}(T) \leq \dfrac{\log(C(d))}{j} + \int_{\T^d} \sum_{i = 1}^{k(x)} \frac{1}{j} \log^+\Big(\chi_i(\apD (T^j)(x))\Big) d \mu(x).
\end{equation}
Letting $j \to \infty$, we find
\begin{equation}
 h_{\mu}(T) \leq \int_{\T^d} \sum_{i = 1}^{k(x)} \lambda^{+}_i(x) m_i(x) \, d \mu(x).
\end{equation}
{This follows from Equation~\eqref{eq:SumOfPositiveLyapunovExpsAsLimit} by dominated convergence since $|\chi_i(\apD (T^j)(x))| \leq \e^{3 j g^{\star}(x)}$, by Proposition~\ref{prop:ApproxDiffPlusBoundPlusSmoothCovering}.}
\end{proof}

\section{Asymptotic Regularity and Bounds on Mixing}
\label{sec:mixingandreg}

In this section, we establish sharp asymptotic regularity estimates and bounds on mixing. 
We prove Theorems~\ref{thm:RLFAsymptoticRegularityThm} and~\ref{thm:RLFAsymptoticMixingThm}, which concern regular Lagrangian flows, as well as corresponding results for general measure-preserving systems with invariant measure equal to $\Leb^d$—see Theorems~\ref{thm:AsymptoticRegularityMPS} and~\ref{thm:AsymptoticMixingMPS} below.

Throughout this section, all measure-preserving systems are assumed to have invariant measure $\Leb^d$. Our regularity results are formulated using the homogeneous log-Sobolev norm $\| \cdot \|_{\dot{H}_{\log}}$ defined in Equation~\eqref{eq:HomogeneousLogSobolevNorm}.

\begin{theorem}\label{thm:AsymptoticRegularityMPS}
Let $(\T^d, \Borel, \Leb^d, T)$ be a measure-preserving system satisfying Assumption~\ref{assump:LusinLip} with some $p \in (1, \infty]$. 
 Let $\rho \in L^\infty\cap\BV(\T^d)$.
 Then
 \begin{equation}\label{eq:AsymptoticRegularityMPSMainEquation}
  \limsup_{n \to \infty} \dfrac{1}{n} \| \rho \circ T^n \|_{\dot{H}_{\log}}^2 \leq C(d) \| \rho \|_{L^\infty}^2\int_{\T^d} \lambda_{\max}(x) \, dx.
 \end{equation}
 
\end{theorem}
This regularity result implies the following lower bound on mixing efficency. This can be expressed in terms of negative Sobolev norms \eqref{eq:H-1} and in terms of the {\it geometric mixing scale} defined as follows: 
\begin{equation*}
 \mix_{\kappa}(\rho) = \inf \left\{ \eps > 0 : \left| \int_{B_r(x)} \rho \, dz \right| \leq \kappa \| \rho \|_{L^{\infty}} \quad \forall x \in \T^d , \, \forall r > \eps \right\}.
\end{equation*}
where $\kappa \in (0,1)$ is a fixed parameter and $\rho \in L^{\infty}(\T^d)$ a mean-free function.

\begin{theorem}\label{thm:AsymptoticMixingMPS}
Let $(\T^d, \Borel, \Leb^d, T)$ be a measure-preserving system satisfying Assumption~\ref{assump:LusinLip} with some $p \in (1, \infty]$ and $g \in L^p(\T^d)$. 
 Let $\rho \in L^\infty\cap\BV(\T^d)$ be mean free, i.e., $\int_{\T^d}\rho  =0$.
 Then
 \begin{align}
  \liminf_{n \to \infty} \dfrac{1}{n} \log \| \rho \circ T^n \|_{\dot{H}^{-1}(\T^d)} &\geq - C(d) \left( \dfrac{\| \rho \|_{L^{\infty}(\T^d)}}{\| \rho \|_{L^{2}(\T^d)}} \right)^2 \int_{\T^d} \lambda_{\max}(x) \, dx; \label{eq:FunctionalMixingInThmForMaps} \\
  \liminf_{n \to \infty} \dfrac{1}{n} \log \mix_{\kappa} \Big(  \rho \circ T^n \Big) &\geq -C(d, \kappa) \int_{\T^d} \lambda_{\max}(x) \, dx \quad \text{ if $\rho \in \{ -1, 1\}$ a.e.} \label{eq:GeometricMixingInThmForMaps}
 \end{align}
\end{theorem}

For the sake of illustration, we outline a simplified argument for Theorem~\ref{thm:RLFAsymptoticRegularityThm} under the stronger assumption that $b \in C^{2}(\mathbb{T}^d; \mathbb{R}^d)$ is autonomous and $\rho_{\initial} \in C^1(\mathbb{T}^d)$. We then discuss the  modifications required to handle the general case.

 Remember that the solution to \eqref{eq:TransportEquation} is given by $\rho(t,x) = \rho_{\initial} \circ X_{0,t}(x)$. Arguing as in \cite[Equation (2.2)]{SBNDN22}, we estimate the exponential separation of trajectories as
 \begin{align*}
  &|X_{0,t}(y) - X_{0,t}(x) - \nabla X_{0,t}(x) (y-x)| \\
  &\quad \leq \e^{\int_0^t |\nabla b(X_{0,s}(x))| \, ds} \int_0^t |b(X_{0,s}(y)) - b(X_{0,s}(x)) - \nabla b (X_{0,s}(x))(X_{0,s}(y)-X_{0,s}(x))| \, ds \\
  &\quad \leq \e^{\int_0^t |\nabla b(X_{0,s}(x))| \, ds} \| b \|_{C^{2}} \int_0^t |X_{0,s}(y) - X_{0,s}(x)|^{2} \, ds
 \end{align*}
 Using the gradient estimate
 \begin{equation}\label{eq:gradient_est}
    |\nabla X_{0,t}(x)| \leq \e^{\int_0^t |\nabla b(X_{0,s}(x))| \, ds}, 
 \end{equation}
we deduce
  \begin{equation}\label{eq:BoundOneInC2Proof}
  |X_{0,t}(y) - X_{0,t}(x)| \leq \e^{\int_0^t |\nabla b(X_{0,s}(x))| \, ds} \left( |y-x| + \| b \|_{C^{2}} \int_0^t |X_{0,s}(y) - X_{0,s}(x)|^{2} \, ds \right).
 \end{equation}
 From this, by taking the square on both side, we get
 \[
  \dfrac{d}{dt} \left[ \left(  |y-x| + \| b \|_{C^{2}} \int_0^t |X_{0,s}(y) - X_{0,s}(x)|^{2} \, ds \right)^{- 1} \right] \geq - \| b \|_{C^{2}} \e^{2 \int_0^t |\nabla b(X_{0,s}(x))| \, ds}.
 \]
 Integrating in an interval $[0,t]$ yields
 \[
  \int_0^t |X_{0,s}(y) - X_{0,s}(x)|^{2} \, ds \leq \dfrac{|y-x|}{\| b \|_{C^{2}}} \left( \dfrac{1}{1 - |y-x| \| b \|_{C^{2}} \e^{2 \int_0^T |\nabla b(X_{0,s}(x))| \, ds}} \right) \leq \dfrac{2 |y-x|}{\| b \|_{C^{2}}} 
 \]
 provided 
\begin{equation}\label{eq:ass_heur}
     |y-x| \| b \|_{C^{2}} \e^{2 \int_0^t |\nabla b(X_{0,s}(x))| \, ds} \leq 1/2.
 \end{equation}
 
  Therefore, under \eqref{eq:ass_heur}, by plugging the inequality above into Equation \eqref{eq:BoundOneInC2Proof}, we obtain
 \[
  |X_{0,t}(y) - X_{0,t}(x)| \leq 3 |y-x| \e^{\int_0^t |\nabla b(X_{0,s}(x))| \, ds}.
 \]
 On the other hand, if \eqref{eq:ass_heur} fails, then $|y-x| \| b \|_{C^{2}} \e^{2 \int_0^t |\nabla b(X_{0,s}(x))| \, ds} \geq 1/2$, therefore
 \[
  |X_{0,t}(y) - X_{0,t}(x)| \leq d \leq C(d) |y-x| \| b \|_{C^{2}} \e^{2 \int_0^t |\nabla b(X_{0,s}(x))| \, ds}.
 \]
 Hence, for all $x , y \in \T^d$ and $t\ge 0$ we have
 \[
  |X_{0,t}(y) - X_{0,t}(x)| \leq C(d) |y-x| (\| b \|_{C^{2}} + 1) \e^{2 \int_0^t |\nabla b(X_{0,s}(x))| \, ds},
 \]
 which implies
 \[
  |\rho(t,x) - \rho(t,y)|  
  \leq C(d) |y-x| \| \nabla \rho_{\initial} \|_{L^{\infty}(\T^d)} (\| b \|_{C^{2}} + 1) \e^{2 \int_0^t |\nabla b(X_{0,s}(x))| \, ds}.
 \]
 Due to Proposition~\ref{prop:LusinLipBoundGivesBoungOnIntegral}, we deduce
 \[
  \limsup_{t \to \infty} \dfrac{1}{t} \| \rho(t,\cdot) \|_{\dot H_{\mathrm{log}}}^2 \leq C(d) \fint_0^1 \int_{\T^d} |\nabla b| \, dx \, dt,
 \]
 a version of Theorem~\ref{thm:RLFAsymptoticRegularityThm}. 
 
 \medskip
 
 The key takeaway from this sketch of proof is that the exponential separation of Lagrangian trajectories can be controlled by the gradient of the flow, provided the velocity field is sufficiently regular. In our sketch, we assumed $C^2$ regularity for simplicity, but in fact, any Holder bound on $\nabla b$ would suffice.
Since the flow satisfies the gradient estimate \eqref{eq:gradient_est}, we deduce that, asymptotically, trajectory separation is governed by the integral $\int_0^t |\nabla b(X_{0,s}(x))|\, ds$, which yields the correct bound. A more refined analysis would replace this with the top Lyapunov exponent, which is the true quantity controlling exponential growth of gradients. This observation is consistent with the statements of Theorems~\ref{thm:RLFAsymptoticRegularityThm} and~\ref{thm:AsymptoticRegularityMPS}.

However, this sketch does not clarify how to weaken the regularity assumption on the velocity field to first-order, as some degree of differentiability for the velocity gradient appears essential to rigorously approximate trajectory separation in terms of the flow gradient. In this regard, the key step is Proposition~\ref{prop:KeyProposition} below whose proof exploits the following two points:
 \begin{itemize}
    \item We leverage weak differentiability to refine the approximate Lipschitz bound in \eqref{eq:LusinLipschitzInequalityInAssumption}. Specifically, up to removing a small exceptional set, Lemma~\ref{lemma:UsefulLemmaAboutTheGradientBeingAGoodApprox} implies that at small scales,
\begin{equation}\label{EstimateByApproximateGradentInIdeas}
	\metric(T(y), T(x)) \leq (\eps + |\apD T(x)|)\, \metric(y, x),
\end{equation}
where $\apD T(x)$ denotes the approximate differential of $T$ at $x$.

    \item Depending on the number of iterations $n$ and each couple of points $(x,y)$, we determine an appropriate scale $r_n(x,y)$.
    Below scale $r_n(x,y)$, we will be able to exploit \eqref{EstimateByApproximateGradentInIdeas} up until the $n$-iteration to estimate $\metric(T^n (y), T^n (x))$. Above scale $r_n(x,y)$, the straightforward estimate $\metric(T^n (y), T^n (x)) \leq d$ is sufficient.
\end{itemize}

\begin{proposition}\label{prop:KeyProposition}
Let $(\T^d, \Borel, \Leb^d, T)$ be a measure-preserving system satisfying Assumption~\ref{assump:LusinLip} with some $p \in (1, \infty]$ and $g \in L^p(\T^d)$ and $(\T^d, \Borel, \Leb^d, S)$ satisfying Assumption~\ref{assump:LusinLip} with $\tilde{g} \in L^1(\T^d)$. 
Let $\eps > 0$ be arbitrary. 
Then there exists a $c = c(\eps, T)$ such that for all $\rho \in  L^{\infty}\cap \BV(\T^d)$ with $\| \rho \|_{L^{\infty}(\T^d)} = 1$ we have
 \begin{align}
 \begin{split}\label{eq:AsymptoticMixingThmPropositionEquation}
  \| \rho \circ T^n \circ S\|_{\dot{H}_{\log}}^2 &\leq C(d) \Bigg( c(\eps, T) + n\eps + n \int_{\T^d} \log|\apD T(x)| \, dx \\
  &\qquad + \eps^{\frac{p-1}{p}} n \left( \| g \|_{L^p(\T^d)}^p + 1 \right) + \| \tilde{g} \|_{L^1(\T^d)} + \| \rho \|_{\BV(\T^d)} \Bigg).
 \end{split}
 \end{align}
\end{proposition}

This is the outline of this section. In Subsection~\ref{subsec:AsympRegularityGeneralAsymptoticMixingThmProof}, we prove Theorem~\ref{thm:AsymptoticRegularityMPS}. 
In Subsection~\ref{subsec:ProofRLFAsymptoticRegularityThm}, we prove Theorem~\ref{thm:RLFAsymptoticRegularityThm}.
As a corollary of Theorem~\ref{thm:AsymptoticRegularityMPS}, we prove Theorem~\ref{thm:AsymptoticMixingMPS} and Theorem~\ref{thm:RLFAsymptoticMixingThm} in Subsection~\ref{subsec:AsympMixingGeneralAsymptoticMixingThmProof}.  
 Both Theorem~\ref{thm:AsymptoticRegularityMPS} and ~\ref{thm:RLFAsymptoticRegularityThm} are based on the Key Proposition~\ref{prop:KeyProposition} which we prove in the final Subsection~\ref{subsec:AsympMixingAsymptoticMixingThmPropositionProof}.  

\subsection{Proof of Theorem~\ref{thm:AsymptoticRegularityMPS}}\label{subsec:AsympRegularityGeneralAsymptoticMixingThmProof}
Without loss of generality, we may assume $\| \rho \|_{L^{\infty}(\T^d)} = 1$.
Let $\eps > 0$ be arbitrary. Let $N$ be sufficiently large so that (which exists by Proposition~\ref{prop:ChainRuleForMPS} and Kingman's sub-additive ergodic theorem \cite[Theorem 3.3.3]{MV16}, or a dominated convergence relying on $g^{\star}$ together with Proposition~\ref{prop:ApproxDiffPlusBoundPlusSmoothCovering})
\begin{equation}\label{eq:ChoiceOfNCloseToLyapunov}
\frac{1}{N} \int_{\T^d} \log |\apD (T^N)| \, dx \leq \int_{\T^d} \lambda_{\max}(x) \, dx + \eps.
\end{equation}
Now let $g_N \in L^1(\T^d)$ be $g_N = \sum_{j = 0}^{N-1} g \circ T^j$ and note that for any $0 \leq k \leq N$, we have 
$$\metric(T^{k}(x) , T^k(y)) \leq (g_N(x) + g_N(y))\metric(x,y) \quad \forall x,y \in \T^d.$$
Now let $n \geq 1$ be an arbitrary integer and note that $n = \lfloor n / N \rfloor N + k$ for some $0 \leq k \leq N-1$.
By Proposition~\ref{prop:KeyProposition}, with $S = T^k$ 
 \begin{align}
 \begin{split}
  \| \rho \circ T^n \|_{\dot{H}_{\log}}^2 &= \| \rho \circ (T^N)^{\lfloor n / N \rfloor} \circ T^k \|_{\dot{H}_{\log}}^2 \leq C(d) \Bigg( c(\eps, T) + \lfloor n / N \rfloor \eps + \lfloor n / N \rfloor \int_{\T^d} \log|\apD (T^N)(x)| \, dx \\
  &\qquad + \eps^{\frac{p-1}{p}} \lfloor n / N \rfloor \left( \| g_N \|_{L^p(\T^d)}^p + 1 \right) + \| g_N \|_{L^1(\T^d)} + \| \rho \|_{\BV(\T^d)} \Bigg).
 \end{split}
 \end{align}
 Dividing by $n$ and taking the $\limsup$ as $n \to \infty$, the two last addends disappear. Invoking \eqref{eq:ChoiceOfNCloseToLyapunov} yields
  \begin{align}
 \begin{split}
  \limsup_{n \to \infty} & \frac{1}{n} \| \rho \circ T^n \|_{\dot{H}_{\log}}^2
  \leq C(d) \Bigg( 2 \eps + \int_{\T^d} \lambda_{\max}(x) \, dx + \eps^{\frac{p-1}{p}} \left( N^{p-1} \| g \|_{L^p(\T^d)}^p + 1 \right) \Bigg).
 \end{split}
 \end{align}
 Since $\eps > 0$ is arbitrary, the proof is concluded.

\subsection{Proof of Theorem~\ref{thm:RLFAsymptoticRegularityThm}}\label{subsec:ProofRLFAsymptoticRegularityThm}
Let $b \in L^1_{\loc}([0, \infty); W^{1,p}(\T^d; \R^d))$ with $p > 1$ be $1$-periodic in time as in the statement of the theorem. 
The solution $\rho_t$ of the transport equation is obtained as $\rho_t = (X_{t})_{\#} \rho_{\initial} = \rho_{\initial} \circ X_{t}^{-1}$ where $X_{t}$ is the RLF at time $t$. Since $b$ is $1$-periodic in time
\begin{equation}
 X_{t,0} = X_{t, \lfloor t \rfloor} \circ X_{\lfloor t \rfloor, 0} =  X_{t - \lfloor t \rfloor, 0} \circ (X_{1})^{\lfloor t \rfloor}
\end{equation}
and therefore
\begin{equation}
 \rho_t = \rho_{\initial} \circ (X_{1, 0}^{-1})^{\lfloor t \rfloor} \circ X_{t - \lfloor t \rfloor, 0}^{-1}.
\end{equation}
By Lemma~\ref{lemma:LusinLipFlows}, there exists $\tilde{g} \in L^p(\T^d)$ with $\| \tilde{g} \|_{L^p(\T^d)} \leq C(d,p) \| \nabla b \|_{L^1([0,1]; L^p(\T^d; \R^{d \times d}))}$ such that
\begin{equation}
 \metric(X_{t,0}^{-1}(x) , X_{t,0}^{-1}(y)) \leq \e^{\tilde{g}(x) + \tilde{g}(y)} \metric(x,y) \quad \forall x,y \in \T^d, \, \forall t \in [0,1].
\end{equation}
Since $(\T^d; \Borel, \Leb^d, X_{1}^{-1})$ satisfies Assumption~\ref{assump:LusinLip}, we may apply Proposition~\ref{prop:KeyProposition} with $T = X_{1}^{-1}$ and $S = X_{t - \lfloor t \rfloor, 0}^{-1}$ to deduce that for all $\eps > 0$
 \begin{align}
 \begin{split}
  \| \rho_t \|_{\dot{H}_{\log}}^2 &\leq C(d) \Bigg( c(\eps, b) + \left\lfloor t \right\rfloor \int_{\T^d} \log(\eps + |\apD (X_{1}^{-1})(x)|) \, dx \\
  &\qquad + \eps^{\frac{p-1}{p}} \left\lfloor t  \right\rfloor \left( \| g \|_{L^p(\T^d)}^p + 1 \right) + \| \tilde{g} \|_{L^1(\T^d)} + \| \rho_{\initial} \|_{\BV(\T^d)} \Bigg).
 \end{split}
 \end{align}
 Thus
 \begin{equation}
  \limsup_{t \to \infty} \frac{1}{t} \| \rho_t \|_{\dot{H}_{\log}}^2 \leq C(d) \int_{\T^d} \log(\eps + |\apD (X_{1}^{-1})(x)|) \, dx + \eps^{\frac{p-1}{p}} \left( \| g \|_{L^p(\T^d)}^p + 1 \right).
 \end{equation}
 Since $\eps$ is arbitrary, we find
 \begin{equation}\label{eq:AsympRegProofRLFConclusionTimeOne}
  \limsup_{t \to \infty} \frac{1}{t} \| \rho_t \|_{\dot{H}_{\log}}^2 \leq C(d) \int_{\T^d} \log |\apD (X_{1}^{-1})(x)| \, dx.
 \end{equation}
 To finish the proof, we rescale time. Replacing $b$ by $(t,x) \mapsto N b(Nt, x)$, the time-one flow becomes the time-$N$ flow of $b$ i.e. $X_N$ and the solution to \eqref{eq:TransportEquation} becomes $(t,x) \mapsto \rho_{Nt}(x)$. Thus by \eqref{eq:AsympRegProofRLFConclusionTimeOne},
 \begin{equation}
 N \limsup_{t \to \infty} \frac{1}{t} \| \rho_t \|_{\dot{H}_{\log}}^2= \limsup_{t \to \infty} \frac{1}{t} \| \rho_{Nt} \|_{\dot{H}_{\log}}^2 \leq C(d) \int_{\T^d} \log |\apD (X_{N}^{-1})(x)| \, dx.
 \end{equation}
 Dividing by $N$ and letting $N \to \infty$
 \begin{equation}
     \limsup_{t \to \infty} \frac{1}{t} \| \rho_t \|_{\dot{H}_{\log}}^2 \leq C(d) \limsup_{N \to \infty}  \int_{\T^d} \dfrac{1}{N} \log |\apD (X_{N}^{-1})(x)| \, dx = C(d) \int_{\T^d} \lambda_{\max}(x) dx.
 \end{equation}
 The last inequality can be obtained by dominated convergence exploiting the maximal ergodic function together with Proposition~\ref{prop:ApproxDiffPlusBoundPlusSmoothCovering} or Proposition~\ref{prop:ChainRuleForMPS} and Kingman's sub-additive ergodic theorem \cite[Theorem 3.3.3]{MV16}.

\subsection{Proof of Theorem~\ref{thm:AsymptoticMixingMPS} and Theorem~\ref{thm:RLFAsymptoticMixingThm}}\label{subsec:AsympMixingGeneralAsymptoticMixingThmProof}

The key ingredient leading to Equation~\eqref{eq:FunctionalMixingInThmForMaps} is the following interpolation inequality, which follows from \cite[Corollary 3.7]{EBQHN21}. Although the result is stated there in $\R^d$, the proof carries over identically to the torus $\T^d$.

\begin{lemma}\label{lemma:MixingBoundByIntegral}
For all $f \in L^2(\T^d)$, it holds that
 \begin{equation}
  \log \left( 2 + \dfrac{\| f \|_{L^2(\T^d)}}{\| f \|_{\dot{H}^{-1}(\T^d)}} \right) \| f \|_{L^2(\T^d)}^2 \leq C(d) \| f \|_{\dot H_{\rm log}}^2.
 \end{equation}
\end{lemma}

The key ingredient leading to Equation~\eqref{eq:GeometricMixingInThmForMaps} is the following lemma:

\begin{lemma}\label{lemma:GeometricMixingScaleLogRegularity}
 Let $\rho \in L^{\infty} \cap \dot{H}_{\log}(\T^d)$ with $\rho \in \{ -1, 1\}$ a.e. If $\mix_{\kappa}(\rho) \leq 1 / 5$ then
 \[
  \mix_{\kappa}(\rho) \geq C \exp \left( - C(d, \kappa) \| \rho \|_{\dot{H}_{\log}(\T^d)}^2 \right).
 \]
\end{lemma}

\begin{proof}
 Denote $\eps = \mix_{\kappa}(\rho)$ and observe that for all $r > \eps$ we have
 \begin{equation}
  1 - \kappa \leq \fint_{B_r(x)} |\rho(z) - \rho(x)| \, dz \quad \forall x \in \T^d.
 \end{equation}
 Hence, for any $x \in \T^d$, we have
 \begin{equation}
  C(d) (1 - \kappa) r^{-1} \leq \frac{1}{r^{d+1}} \int_{B_r(x)} |\rho(z) - \rho(x)| \, dz \quad \forall r > \eps.
 \end{equation}
 Integrating from $\eps$ to $1 / 5$ with respect to $r$, we find
 \begin{equation}\label{eq:InProofOfGeometricMixingScaleLogRegularity-A}
  C(d, \kappa) (\log(1/5) - \log(\eps)) \leq \int_{\eps}^{1/ 5} \frac{1}{r^{d+1}} \int_{B_r(x)} |\rho(z) - \rho(x)| \, dz \, dr.
 \end{equation}
 The right-hand side of Equation~\eqref{eq:InProofOfGeometricMixingScaleLogRegularity-A} equals
 \begin{equation}
  \underbrace{\int_{\eps}^{1/ 5} \frac{1}{r^{d+1}} \int_{B_{\eps}(x)} |\rho(z) - \rho(x)| \, dz \, dr}_{=: I} + \underbrace{\int_{\eps}^{1/ 5} \frac{1}{r^{d+1}} \int_{\eps}^{r} \int_{\partial B_{s}(x)} |\rho(z) - \rho(x)| \, d\Haus^{d-1}(z) \, ds \, dr}_{=: II}.
 \end{equation}
 For these two terms, we have
 \begin{align*}
 I &\leq C(d) \int_{B_{\eps}(x)} \frac{|\rho(z) - \rho(x)|}{\eps^d} \, dz \leq C(d) \int_{B_{\eps}(x)} \frac{|\rho(z) - \rho(x)|}{|z-x|^d} \, dz; \\
 II &= \int_{\eps}^{1 / 5} \int_{s}^{1/5} \frac{1}{r^{d+1}} \int_{\partial B_{s}(x)} |\rho(z) - \rho(x)| \, d\Haus^{d-1}(z) \, dr \, ds \leq C(d) \int_{\eps}^{1 / 5} \int_{\partial B_{s}(x)} \frac{|\rho(z) - \rho(x)|}{s^d} \, d\Haus^{d-1}(z) \, ds.
 \end{align*}
 Thus, by Equation~\eqref{eq:InProofOfGeometricMixingScaleLogRegularity-A} 
 \begin{equation}\label{eq:InProofOfGeometricMixingScaleLogRegularity-B}
  \log \left( \frac{1}{5 \eps} \right) \leq C(d, \kappa) (I + II) \leq C(d, \kappa) \int_{B_{1/5}(x)} \frac{|\rho(z) - \rho(x)|}{|z-x|^d} \, dz.
 \end{equation}
 Since $\rho \in \{ -1, 1\}$ a.e., we have $|\rho(z) - \rho(x)|^2 = 2 |\rho(z) - \rho(x)|$ $\Leb^{2d}$-a.e. and thus integrating Equation~\eqref{eq:InProofOfGeometricMixingScaleLogRegularity-B} with respect to $x$ yields
  \begin{equation}
  \log \left( \frac{1}{5 \eps} \right) \leq C(d, \kappa) \| \rho \|_{\dot{H}_{\log}}^2.
 \end{equation}
 Taking the exponential finishes the proof.
\end{proof}

\begin{proof}[Proof of Theorem~\ref{thm:AsymptoticMixingMPS}]
From Lemma \ref{lemma:MixingBoundByIntegral} and Theorem \ref{thm:AsymptoticRegularityMPS}, we get 
\begin{equation}
\begin{split}
\limsup_{n \to \infty} \frac{1}{n} \log \left( 2 + \dfrac{\| \rho \circ T^n \|_{L^2(\T^d)}}{\| \rho \circ T^n \|_{\dot{H}^{-1}(\T^d)}} \right) \| &\rho \circ T^n \|_{L^2(\T^d)}^2 
 \leq C(d) \limsup_{n \to \infty} \frac{1}{n} \| \rho \circ T^n\|_{\dot{H}_{\log}}^2
\\& \le C(d) \| \rho \|_{L^\infty}^2 \int_{\T^d} \lambda_{\rm max}(x) dx.
\end{split}
\end{equation}
Since $\|\rho\circ T^n \|_{L^2} = \| \rho \|_{L^2}$,  
Equation~\eqref{eq:FunctionalMixingInThmForMaps} follows.
Lemma~\ref{lemma:GeometricMixingScaleLogRegularity} implies
\begin{equation}
 \log \left( \dfrac{ \mix_{\kappa}(\rho \circ T^n)}{C} \right) \geq - C(d, \kappa) \| \rho \circ T^n \|_{\dot{H}_{\log}(\T^d)}^2
\end{equation}
Dividing by $n$ and taking the $\liminf$ as $n \to \infty$, we deduce Equation~\eqref{eq:GeometricMixingInThmForMaps} by Theorem~\ref{thm:AsymptoticRegularityMPS}.
\end{proof}

A completely analogous argument applies also to Theorem~\ref{thm:RLFAsymptoticMixingThm}.

\subsection{Proof of Proposition~\ref{prop:KeyProposition}}\label{subsec:AsympMixingAsymptoticMixingThmPropositionProof}
We will use the following result from \cite[Proposition 2.16]{EBQHN21}. Although it is stated for $\R^d$ in \cite{EBQHN21}, the proof carries over identically to the torus $\T^d$.

\begin{proposition}\label{prop:LusinLipBoundGivesBoungOnIntegral}
 Let $f \in L^1\cap L^{\infty}(\T^d)$ and $g \in L^1(\T^d)$ satisfy
 \begin{equation}
  |f(x) - f(y)| \leq \e^{g(x) + g(y)} d(x,y) \quad \forall x,y \in \T^d.
 \end{equation}
 Then 
 \begin{equation}
     \int_{\mathbb{T}^d} \left( \int_{B_{1/5}(0)} \dfrac{1 \wedge |f(x+h) - f(x)|^2}{|h|^d} dh\right) dx
     \le C(d)\left( \| g \|_{L^1(\T^d)} + \| f \|_{L^1(\T^d)} \right).
 \end{equation}
 In particular, if $\| f \|_{L^\infty} \le 1$, then
 $\| f \|_{\dot{H}_{\log}}^2 \leq C(d) \left( \| g \|_{L^1(\T^d)} + \| f \|_{L^1(\T^d)} \right)$.
\end{proposition}

Let $\eps > 0$ be arbitrary as in the statement of Proposition~\ref{prop:KeyProposition}. 
By Lemma~\ref{lemma:UsefulLemmaAboutTheGradientBeingAGoodApprox}, there exists a set $A \subseteq \T^d$ such that $\Leb^d (\T^d \setminus  A) < \eps$ and for all $x,y \in A$ with $|x-y| < \delta$, we have
\begin{equation}
 |T(y) - T(x) - \apD T(x) (y-x)| < \eps |x-y|.
\end{equation}
Since $\rho \in \BV(\T^d)$, there exists $h \in L^1(\T^d)$ with $\| h \|_{L^1(\T^d)} \leq C(d) \| \rho \|_{\BV(\T^d)} + 1$ such that $|\rho(x) - \rho(y)| \leq \e^{h(x) + h(y)} d(x,y)$ for all $x,y \in \T^d$. 
Thus,
\begin{equation}\label{eq:AsympMixingPropClaimFirstBoundUsingBV}
 |(\rho \circ T^n \circ S)(x) - (\rho \circ T^n \circ S)(y)| \leq \e^{h(T^n(S(x))) + h(T^n(S(y)))} \metric((T^n \circ S)(x) , (T^n \circ S)(y)) \quad \forall x,y \in \T^d.
\end{equation}
We now claim the following: 
\\
\textbf{Claim: }
For all $n \geq 1$, there exists $\ell_n \in L^1(\T^d)$ with 
\begin{equation}\label{eq:AsympMixingPropClaimEquationOne}
 \| \ell_n \|_{L^1(\T^d)} \leq \log(\sfrac{d}{\delta}) + n\eps + n \int_{\T^d} \log |\apD T(x)| \, dx + \eps^{\frac{p-1}{p}} n \left( \| g \|_{L^p(\T^d)}^p + 1 \right)
\end{equation}
such that
\begin{equation}\label{eq:AsympMixingPropClaimEquationTwo}
 \metric(T^n (x) , T^n (y)) \leq \e^{\ell_n(x) + \ell_n(y)} \metric(x,y) \quad \forall x,y \in \T^d.
\end{equation}
The proof of the claim is postponed.
From the claim and \eqref{eq:AsympMixingPropClaimFirstBoundUsingBV}, we deduce
\begin{align*}
 &|(\rho \circ T^n \circ S)(x) - (\rho \circ T^n \circ S)(y)| \\
 &\qquad \leq \exp( h(T^n(S(x))) + h(T^n(S(y))) + \ell_n(S(x)) + \ell_n(S(y)) + \tilde{g}(x) + \tilde{g}(y)) \metric(x,y) \quad \forall x,y \in \T^d.
\end{align*}
From Proposition~\ref{prop:LusinLipBoundGivesBoungOnIntegral}, it follows that
\begin{align*}
 \| \rho \circ T^n \circ S \|_{\dot{H}_{\log}}^2 
 &\leq C(d) \| h \|_{L^1(\T^d)} + \| \ell_n \|_{L^1(\T^d)} + \| \tilde{g} \|_{L^1(\T^d)} + \| \rho \|_{L^1(\T^d)} \\
 &\leq C(d) \| \rho \|_{\BV(\T^d)} + \| \ell_n \|_{L^1(\T^d)} + \| \tilde{g} \|_{L^1(\T^d)} + C(d).
\end{align*}
Plugging in Equation \eqref{eq:AsympMixingPropClaimEquationOne} gives \eqref{eq:AsymptoticMixingThmPropositionEquation} and hence the proof is finished.
It only remains to prove the claim. 
\begin{proof}[Proof of Claim.]
Define $\Theta \colon \T^d \times \T^d \to \R$ by
 \begin{equation}
  \Theta(x,y)
  =
  \left\{
  \begin{array}{ll}
   \eps + |\apD T(x)| & \text{if $x,y \in A$;} \\
   \e^{g(x) + g(y)} & \text{otherwise.}
  \end{array}
  \right.
 \end{equation}
 We note that $d(T(x) , T(y)) \leq \Theta(x,y) d(x,y)$ for all $x,y \in \T^d$ with $|x-y| < \delta$.
 For any integer $n \geq 1$, we define $\Theta_n \colon \T^d \times \T^d \to \R$ as
 \begin{equation}
  \Theta_n(x,y) = \prod_{j = 0}^{n-1} \Theta(T^j(x), T^j(y)).
 \end{equation}
 Note that for fixed $x,y \in \T^d$, $ \Theta_n(x,y)$ is increasing in $n$.
 We observe that if $\metric(T^j(x) , T^j(y)) < \delta$ for all $j = 0, \ldots, n-1$, then
 \begin{equation}\label{eq:AsympMixingPropClaimBoundOnClosePoints}
  \metric(T^n(x) , T^n(y)) \leq \Theta_n(x,y) \metric(x,y) \quad \forall x,y \in \T^d.
 \end{equation}
 Motivated by this observation, we define a function $r_n \colon \T^d \times \T^d \to \R$ by $r_n(x,y) = \frac{\delta}{ \Theta_n(x,y)}$. From the observation above, we deduce that if $\metric(x,y) < r_n(x,y)$, then $\metric(T^n(x) , T^n(y)) \leq \Theta_n(x,y) \metric(x,y)$. Moreover, if $d(x,y) > r_n(x,y)$ then
 \begin{equation}\label{eq:AsympMixingPropClaimBoundOnDistantPoints}
  \metric(T^n(x) , T^n(y)) \leq d \leq d \frac{\metric(x,y)}{r_n(x,y)} \leq \frac{d}{\delta} \Theta_n(x,y) \metric(x,y).
 \end{equation}
 Hence, combining \eqref{eq:AsympMixingPropClaimBoundOnClosePoints} and \eqref{eq:AsympMixingPropClaimBoundOnDistantPoints}, we find
 \begin{equation}
   \metric(T^n(x) , T^n(y)) \leq \frac{d}{\delta} \Theta_n(x,y) \metric(x,y) \quad \forall x,y \in \T^d.
 \end{equation}
 In order to find a function $\ell_n$ as in Equation~\eqref{eq:AsympMixingPropClaimEquationTwo}, we now provide an upper bound for $\log \Theta_n(x,y)$ with a function of the form $\ell_n(x)+ \ell_n(y)$ for all $x,y \in \T^d$.
 For any $x \in \T^d$, we define the set $S_n(x)$ as 
 \begin{equation}
  S_n(x) = \{ 0 \leq j \leq n-1 : T^j(x) \in A \} \subseteq \N_{\geq 0}.
 \end{equation}
 Then,
 \begin{align}
 \begin{split}\label{eq:AsympMixingPropClaimBoundOnLogTheta}
  \log \Theta_n(x,y) &= \sum_{j = 0}^{n-1} \log \Theta (T^j(x), T^j(y)) \\
   &= \sum_{j \in S_n(x) \cap S_n(y)} \log \Theta (T^j(x), T^j(y)) + \sum_{j \in S_n(x)^c \cup S_n(y)^c} \log \Theta (T^j(x), T^j(y))  \\
   &= \sum_{j \in S_n(x) \cap S_n(y)} \log (\eps + |\apD T(T^j (x))|) + \sum_{j \in S_n(x)^c \cup S_n(y)^c} \left( g(T^j (x)) + g(T^j (y)) \right) \\
 &\leq \sum_{j = 0}^{n-1} \log (\eps + |\apD T(T^j (x))|) + \sum_{j = 0}^{n-1} \left( g(T^j (x)) + g(T^j (y)) \right) \left( \mathbbm{1}_{A^c}(T^j (x)) + \mathbbm{1}_{A^c}(T^j (y)) \right).
 \end{split}
 \end{align}
 For the each term in the last sum, by Young's inequality we have for all $j=0,...,n-1$
 \begin{align*}
     &\left( g(T^j (x)) + g(T^j (y)) \right) \left( \mathbbm{1}_{A^c}(T^j (x)) + \mathbbm{1}_{A^c}(T^j (y)) \right) \\
     &\qquad \leq \eps^{\beta p} \left( g(T^j (x)) + g(T^j (y)) \right)^p + \eps^{-\beta p^{\prime}} \left( \mathbbm{1}_{A^c}(T^j (x)) + \mathbbm{1}_{A^c}(T^j (y)) \right)^{p^{\prime}} \\
     &\qquad \leq C(p) \Big( \eps^{\beta p} g^p(T^j (x)) + \eps^{\beta p} g^p(T^j (y)) + \eps^{-\beta p^{\prime}} \mathbbm{1}_{A^c}(T^j (x)) + \eps^{-\beta p^{\prime}} \mathbbm{1}_{A^c}(T^j (y)) \Big).
 \end{align*}
where $\beta \in \R$ is still to be selected (depending on $p$).
 Hence, with $\ell_n \colon \T^d \to \R$ defined by
 \begin{equation}
  \ell_n(x) = \log(d/\delta) + \sum_{j = 0}^{n-1} \Big( \log (\eps + |\apD T(T^j (x))|) + C(p)\eps^{\beta p} g^p(T^j (x)) + C(p)\eps^{- \beta p^{\prime}}  \mathbbm{1}_{A^c}(T^j (x)) \Big)
 \end{equation}
 we have by the previous two inequalities
 \[
 \log \Theta_n(x,y) \leq \ell_n(x) + \ell_n(y).
 \]
 By \eqref{eq:AsympMixingPropClaimBoundOnClosePoints} and the previous inequality we get that $\metric(T^n(x) , T^n(y)) \leq \e^{\ell_n(x) + \ell_n(y)} \metric(x,y)$ for all $x,y \in \T^d$, namely. 
 It only remains to estimate the $L^1$ norm of $\ell_n$. We have
 \begin{align*}
  \| \ell_n \|_{L^1} 
 &\leq \log(\sfrac{d}{\delta}) + n \int_{\T^d} \log(\eps + |\apD T(x)|) \, dx  + C(p) \sum_{j = 0}^{n-1} \Big( \eps^{\beta p} \left\| g^p \circ T^j \right\|_{L^1} + \eps^{- \beta p^{\prime}} \left\| \mathbbm{1}_{A^c} \circ T^j \right\|_{L^1} \Big).
 \end{align*}
We estimate  the second addend in the right-hand side using the fact that $\log(x + y) \leq x + \log y$ whenever $x \geq 0$ and $y \geq 1$. For the two last addends, we exploit that $T$ is measure-preserving so that by selecting $\beta = \frac{p-1}{p^2}$, we obtain
 \begin{equation}
  \| \ell_n \|_{L^1(\T^d)} \leq \log(\sfrac{d}{\delta}) + n \eps +n \int_{\T^d} \log(|\apD T(x)|) \, dx + \eps^{\frac{p-1}{p}} n \left( \| g \|_{L^p(\T^d)}^p + 1 \right).
  \end{equation}

\end{proof}

\bibliographystyle{plain}
\bibliography{biblio}
 
\end{document}